\documentclass[a4paper, reqno]{amsart}

\calclayout
\makeatletter
\def\ifdraft{\ifdim\overfullrule>\z@
        \expandafter\@firstoftwo
    \else
        \expandafter\@secondoftwo
    \fi}
\makeatother 

\def\clap#1{\hbox to 0pt{\hss#1\hss}}

\usepackage{amssymb} 
\usepackage{enumitem} 
\usepackage{calc}
\usepackage{tikz-cd}
\usepackage{mathtools} 
\usepackage{calrsfs}   
\usepackage[final]{hyperref} 
\usepackage{cleveref}
\usepackage{environ}
\usepackage{float}
\usepackage{tensor}
\newcommand*{\doi}[1]{doi:\href{https://doi.org/#1}{#1}}
\newcommand*{\doii}[2]{doi:\href{https://doi.org/#1}{#2}}
\usepackage{ifthen} 
\usepackage{graphicx}
\graphicspath{ {./figures/} }

\usepackage{xcolor}

\usetikzlibrary{decorations.markings,arrows.meta}

\tikzset{PO/.style = {dashed,very near start, commutative diagrams/phantom, "\lrcorner", /tikz/commutative diagrams/.cd,dr}} 

\ifdraft{
    \usepackage{silence} 
    \usepackage[inline]{showlabels} 
    
}

\newtheorem{theorem}{Theorem}[section]

\newtheorem{lemma}[theorem]{Lemma}
\newtheorem{corollary}[theorem]{Corollary}

\theoremstyle{definition}
\newtheorem*{definitionplain}{Definition}
\newtheorem{definition}[theorem]{Definition}

\newtheorem{setup}[theorem]{Setup}

\let\mod\relax
\let\Im\relax

\DeclareMathOperator{\add}{add}
\DeclareMathOperator{\Coker}{Coker}
\DeclareMathOperator{\CM}{CM}
\DeclareMathOperator{\D}{D}
\DeclareMathOperator{\End}{End}
\DeclareMathOperator{\Ext}{Ext}

\DeclareMathOperator{\Hom}{Hom}
\DeclareMathOperator{\id}{id}
\DeclareMathOperator{\Im}{Im}
\DeclareMathOperator{\Ker}{Ker}
\DeclareMathOperator{\mod}{mod}

\DeclareMathOperator{\op}{op}
\DeclareMathOperator{\pr}{pr}
\DeclareMathOperator{\proj}{proj}
\DeclareMathOperator{\rad}{rad}
\DeclareMathOperator{\RHom}{RHom}
\DeclareMathOperator{\Spec}{Spec}
\DeclareMathOperator{\Sub}{Sub}
\DeclareMathOperator{\uSub}{\underline{Sub}}
\DeclareMathOperator{\Tor}{Tor}
\DeclareMathOperator{\Tot}{Tot}
\DeclareMathOperator{\thick}{thick}
\DeclareMathOperator*{\tens}{\otimes}
\DeclareMathOperator{\uCM}{\underline{\CM}}
\DeclareMathOperator{\uEnd}{\underline{\End}}

\DeclareMathAlphabet{\pazocal}{OMS}{zplm}{m}{n}
\newcommand{\cC}{\mathcal{C}}
\newcommand{\cE}{\mathcal{E}}
\newcommand{\cP}{\mathcal{P}}
\newcommand{\cT}{\mathcal{T}}

\newcommand{\pC}{\pazocal{C}}
\newcommand{\pD}{\pazocal{D}}
\newcommand{\pK}{\pazocal{K}}

\newcommand{\fa}{\mathfrak{a}}

\newcommand{\ZZ}{\mathbb{Z}}
\newcommand{\CC}{\mathbb{C}}
\newcommand{\NN}{\mathbb{N}}

\newcommand{\uA}{{\underline{A}}}
\newcommand{\uB}{{\underline{B}}}

\DeclarePairedDelimiterX\Set[1]{\lbrace}{\rbrace}{ #1 }

\newcommand\ltensor{\tens^{\normalfont\text{L}}}
\newcommand\simto{\stackrel{\sim}{\smash{\longrightarrow}\rule{0pt}{0.4ex}}}

\newlength\tindent
\renewcommand{\indent}{\hspace*{\tindent}}

\title{Derived equivalences of Self-injective 2-Calabi--Yau Tilted Algebras} 

\author{Anders S. Kortegaard}
\address{Department of Mathematics, Aarhus University, Ny Munkegade 118, 8000 Aarhus C, Denmark}
\email{kortegaard@math.au.dk}

\begin{document}

\begin{abstract}
    Consider a $k$-linear Frobenius category $\cE$ with a projective generator such that
    the corresponding stable category $\cC$ is 2-Calabi--Yau, Hom-finite with split idempotents.
    Let $l,m\in\cC$ be maximal rigid objects with self-injective endomorphism algebras.
    We will show that their endomorphism algebras $\cC(l,l)$ and $\cC(m,m)$ are derived equivalent.
    Furthermore we will give a description of the two-sided tilting complex which induces this derived equivalence.
\end{abstract}
\maketitle

\section{Introduction}

In \cite{august2020tilting} August showed that given two objects $M,N\in \uCM(R)$ with $\Spec R$ being a complete local isolated cDV singularity, such that $M$ and $N$ are maximal rigid objects connected through
a number of mutations, the contraction algebras $\underline{\End}(M)$ and $\underline{\End}(N)$ are derived equivalent.
Note that in this setting $\uEnd(M)$ and $\uEnd(N)$ are symmetric algebras.
\medbreak

In this paper we generalize the result mentioned above to the setting of a more general Frobenius category than $\CM(R)$.
Our general course of action and a number of the proofs will be based on those in \cite{august2020tilting}.
We shall use a result from \cite{zhou2011maximal}
that will give a conflation which is able to replace the exchange sequences you would get from mutations.
This will allow us to prove that $\uEnd(M)$ and $\uEnd(N)$ are derived equivalent if they are self-injective and $M$ and $N$ are maximal rigid objects.
This will therefore lead to significantly more general results, allowing for categories without the condition $\Sigma^2 \cong \id$ and categories that may have infinitely many maximal-rigid objects.
\medbreak

\textbf{Two-sided tilting complex}. The definition of a tilting complex was introduced by Rickard in \cite{rickard1989morita}. He showed that given two derived equivalent algebras,
there exists a tilting complex inducing such an equivalence.
\medbreak

Let $k$ be an algebraically closed field, $A$ a $k$-algebra. Let $\proj A$ denote the category of f.g. projective left $A$-modules, $\pK^b(\proj A)$ its bounded homotopy category.
\begin{definitionplain}
    A complex $T\in\pK^b(\proj A)$ is called a \emph{tilting complex} if the following are satisfied.
    \begin{itemize}
        \item $\Hom(T, T[i]) = 0$ for all $i\neq 0$.
        \item $\thick(T) = \pK^b(\proj A),$ where $\thick(T)$ indicates the thick closure of $\add(T)$.
    \end{itemize}
\end{definitionplain}

The notion of a two-sided tilting complex is also due to Rickard \cite{rickard1991derived}.
We use the following form of the definition due to Keller \cite[8.1.4]{Keller1998}.
Let unadorned tensor products be over $k$ and let $\pD$ denote the derived category.

\begin{definition}
    \label{def:tilting-complex}
    Let ${}_{A}T_{B}\in\pD(A\otimes B^{\op})$ be a complex of left $A$- and right $B$-modules.
    Let ${}_{A}T$ (resp. $T_{B}$) be ${}_{A}T_{B}$ seen as a complex in $\pD(A)$ (resp. $\pD(B^{\op})$).
    $\tensor[_A]{T}{_B}$ is a \emph{two-sided tilting complex} if the following are satisfied:
    \begin{enumerate}
        \item The canonical map $A\rightarrow\Hom_{\pD(B^{\op}) } (T_B,T_B)$ is bijective and $\Hom_{D(B^{\op})}(T_B, T_B[i]) = 0$ for $i\neq 0$.
        \item $T_B$ is quasi-isomorphic to a complex in $\pK^b(\proj B^{\op})$.
        \item $\thick(T_B) = \pK^b(\proj B^{\op})$.
    \end{enumerate}
\end{definition}

\renewcommand{\thetheorem}{\Alph{theorem}}
\setcounter{theorem}{0}

Let $\cE$ be a $k$-linear Frobenius category, with projective-injective objects on the form $\add(r)$, for some object $r\in\cE$.
Then given two maximal rigid objects $l,m\in\cC$, the complex $T=\cE(l,m)$
is a two-sided tilting complex in $\pD(\cE(m,m)\otimes \cE(l,l)^{op})$ (see \cite[prop 5.1]{jorgensen2019green}), making $A=\cE(l,l)$ and $B=\cE(m,m)$ derived equivalent.
Looking at the stable category $\cC$ of $\cE$, a similar choice of the module $\cC(l,m)$ does not necessarily give a tilting module, and it is not necessarily true that $\underline{A}=\cC(l,l)$ and $\underline{B}=\cC(m,m)$ are derived equivalent.
However, we are able to prove the following main result.

\begin{theorem}[\Cref{cor:Ptilting}]
    \label{thm:intro:main}
    Let $\cE$ be $k$-linear Frobenius category with projective objects $\add(r)$ for some $r\in \cE$,
    such that the associated stable category $\cC \coloneqq \underline{\cE}$ is 2-CY and Hom-finite with split idempotents.

    \indent Let $l,m\in \cC$ be maximal rigid objects,
    $A=\cE(l,l)$, $\uA = \cC(l,l)$, $B=\cE(m,m)$, $\uB = \cC(m,m)$, $T = \cE(l,m)$.
    Assume that $\uA$ and $\uB$ are self-injective then there exists a two-sided tilting complex of $\underline{B}\otimes \underline{A}^{\op}$-modules 
    \begin{equation*}
        {}_{\underline{B}}\mathcal{T}_{\underline{A}} 
        =  \left(\tensor[_\uB]{\uB}{_B}\ltensor_B \tensor[_B]{T}{_A} \ltensor_A \tensor[_A]{\uA}{_\uA} \right)_{\subseteq 1},
    \end{equation*}

    making $\underline{A}$ and $\underline{B}$ derived equivalent. 
    The subscript $\subseteq 1$ denotes the soft truncation to homological degrees $\leq 1$.
\end{theorem}

In the last section we will focus on two examples.
The first example is on cluster-tilting objects in the cluster category $\pC(D_{2n})$.
In \cite{ringel2008self} Ringel has listed the cluster-tilting objects of $\pC(D_{2n})$ with self-injective endomorphism algebras. 
It has already been showed in \cite[lem. 4.5]{bastian2014towards} that these endomorphism algebras are derived equivalent.
This result was achieved by supplying a tilting complex ad hoc. 
We will use our result to recover this tilting complex by manually calculating the tilting complex $\mathcal{T}_\uA$. 

\indent The second example will be on a class of examples based on 
Postnikov diagrams.
We will use the work of Pasquali \cite{pasquali2020self}, in which he describes how
reduced and symmetric $(k,n)$-Postnikov diagrams give rise to cluster tilting objects with self-injective endomorphism algebras.
This will lead to the following result:

\begin{corollary}[\Cref{cor:postnikov_example}]
    Let $k,n\in \NN$, with $k<n$. 
    Let $\hat B$ be the completion of the so-called boundary algebra (see section 4). 
    Let $D,D'$ be two symmetric and reduced $(k,n)$-Postnikov diagrams, with
    associated cluster tilting objects $T, T'$ (resp.) in $\uCM(\hat B)$ the stable category of Cohen–Macaulay modules.
    Then the self-injective algebras $\underline{\End}(T)$ and $\underline{\End}(T')$ are derived equivalent.
\end{corollary}

In addition to the work of August, let us mention the following related work.
In \cite{dugas2015construction} Dugas construct pairs of derived equivalent algebras using triangles coming from approximations,
which he then applied in the case of symmetric algebras.
Work by Mizuno in \cite{mizuno2019derived} constructs derived autoequivalences of preprojective algebras of Dynkin type which in particular are self-injective.
In \cite{asashiba1999derived} Asashiba classifies all derived equivalences of self-injective representation finite algebras.

\section{Preliminaries}
\renewcommand{\thetheorem}{\arabic{section}.\arabic{theorem}}
\setcounter{theorem}{0}

\begin{setup}
    \label{set:EC}
Let $k$ be an algebraically closed field. Let $\cE$ be a Frobenius category
such that the class of projective objects is $\add(r)$ for some $r\in\cE$. Let $\cC \coloneqq \underline\cE$ be the associated stable category.
We will assume that $\cC$ is a 2-Calabi--Yau, Hom-finite category with split idempotents.
Observe that $\cC$ has the same objects as $\cE$ but different morphisms.
\end{setup}
\medbreak

It is well-known that $\cC$ is a triangulated category, whose suspension functor will be denoted $\Sigma$.

\pagebreak
\begin{definition}\leavevmode
    \begin{itemize}
        \item $x\in \cC$ is called \emph{rigid} if $\cC(x,\Sigma x)=0$.
        \item $x\in \cC$ is called \emph{maximal rigid} if it is rigid, and $\cC(x\oplus y,\Sigma(x \oplus y))=0$ implies $y\in \add(x)$.
        \item $x\in \cC$ is called \emph{good} if $x \cong x' \oplus r$ in $\cE$, for some object $x'\in\cE$.
    \end{itemize}
\end{definition}

Notice that due to the assumption that there is a projective generator $r$, every object in $\cC$ has a representative in $\cE$ which is good.
\medbreak

The following result is due to Zhou and Zhu \cite{zhou2011maximal}.
It generalizes a similar result from \cite{geiss2006rigid}, which then can be applied to our setup. 
We will use it to construct conflations which will then connect maximal rigid objects in a way that can “replace” exchange sequences of mutations.

\begin{theorem}[{\cite[cor. 2.5]{zhou2011maximal}}]
    \label{thm:ZZ_approx}
    Let $x\in\cC$ be maximal rigid, and let $y\in\cC$ be rigid,
    then $y\in \add(x)*\add(\Sigma x)$, i.e. there exists a triangle
    \begin{equation*}
    \begin{tikzcd}
        x_1\rar&x_0\rar&y\rar&\Sigma x_1,
    \end{tikzcd}
    \end{equation*}
    with $x_i\in \add(x)$.
\end{theorem}

The following lemma is a collection of useful results when working in the context of \Cref{set:EC}. See also \cite[lem. A.1]{jorgensen2019green}.
\begin{lemma}
    \label{lem:wif}
    Let $x,y,z\in \cE$. Define $A\coloneqq \cE(x,x)$ and $\underline{A}\coloneqq \cC(x,x)$. 
    \begin{enumerate}[label=(\alph*)]
        \item If $x\cong 0$ in $\cC$, then $x$ is a projective object in $\cE$.\label{lem:item:wif5}
        \item For each triangle $x\xrightarrow{\underline{u}} y\rightarrow z\rightarrow \Sigma x$ in $\cC$ there is a conflation $0\rightarrow x\rightarrow y'\rightarrow z'\rightarrow 0$ in $\cE$
            such that $y\cong y'$ and $z\cong z'$ in $\cC$. \label{lem:item:wif1}
        \item $x\cong y$ in $\cC$ if and only if there exist projective objects $p,p'\in \cE$
            such that $x\oplus p \cong y\oplus p'$ in $\cE$.\label{lem:item:wif2}
        \item If $x\in \add(y)$ or $z\in \add(y)$ then composition of morphisms induces a $k$-linear bijection
            $\cE(y,z)\otimes_B \cE(x,y)\rightarrow\cE(x,z)$ which is natural in $x,z$, where $B=\cE(y,y)$.\label{lem:item:wif4}
        \item If $\tilde x\in\add(x)$ then composition of morphisms induces a $k$-linear bijection
            \(\cC(x,y)\otimes_A\cE(\tilde x,x)\rightarrow \cC(\tilde x,y )\) which is natural in $\tilde x, y$.
            \label{lem:item:wif3}
        \item If $y\in\add(x)$ then the canonical map
            $\cC(x,-):\cC(y,z)\to \Hom_{\underline{A}}(\cC(x,y),\cC(x,z))$ is a bijection. \label{lem:item:wif7}
        \item If $x$ is maximal rigid and $y,z$ are rigid then the map 
            $\cC(y,z)\to \Hom_{\underline{A}}(\cC(x,y),\cC(x,z))$ is surjective. \label{lem:item:wif6}
    \end{enumerate}
\end{lemma}
\begin{proof}\leavevmode
    \ref{lem:item:wif5}
    If $x\cong 0$ in $\cC$, then $\underline{\id_x} = 0$ and therefore $\id_x$ factors through a projective object $\id_x= p'p:x\rightarrow P \rightarrow x$ in $\cE$. 
    This makes $x$ projective (see \cite[cor 11.4]{buhler2010exact}).

    \medbreak

   \ref{lem:item:wif1} Since $\cE$ is a Frobenius category, there are enough injectives. Hence there is an inflation $\alpha:x\rightarrow I$, with $I$ being an injective object.
   It follows that 
   \[
       \begin{pmatrix} u\\\alpha \end{pmatrix}: x \rightarrow y\oplus I,
   \]
    is an inflation. Thus there is a conflation $0\rightarrow x\rightarrow y\oplus I \rightarrow z'\rightarrow 0$. Since $I$ is injective, $y\oplus I\cong y$ in $\cC$, and
    it then follows that $z\cong z'$ in $\cC$.

    \medbreak

    \ref{lem:item:wif2} The `if' part is straightforward. For the `only if' part assume that $x\cong y$ in $\cC$.
    This implies the existence of a triangle $x\rightarrow y\rightarrow 0 \rightarrow \Sigma x$, which by \ref{lem:item:wif1} is induced by a conflation $0\rightarrow x\rightarrow y\oplus I\rightarrow P\rightarrow 0$,
    with $P\cong 0$ in $\cC$. Now $P$ is a projective object by \ref{lem:item:wif5}, giving that the conflation splits. Hence $x\oplus P \cong y\oplus I$.

    \medbreak
    \ref{lem:item:wif4}
    Consider the morphism $\tilde\circ(-) : \cE(y,z)\otimes_B \cE(x,y) \rightarrow  \cE(x,z)$ induced by composition.
    We need to check that this is bijective. We show the case $x\in \add(y)$, the case for $z\in\add(y)$ is similar.
    Assume that $x\in\add(y)$. This means that there exist diagrams

    \begin{equation*}
    \begin{tikzcd}
        x \rar["\eta"{anchor=south},bend left] & y^n \ar[l,"\nu"{anchor=north}, bend left]\rar["\pi_i", bend left] & y \ar[l,bend left, "\iota_i"],
    \end{tikzcd}
    \end{equation*}
    such  that $\nu \eta = \id _x$, and where $\pi_i$ is the projection to the $i$'th component, and $\iota_i$ is the inclusion of the $i$'th component. 

    \indent To show surjectivity, let $\psi\in\cE(x,z)$. Since $\id_{y^n} = \sum_i \iota_i \pi_i$, the element mapped to $\psi$ can be found as follows:
    \begin{align*}
        \psi 
        &= \psi \nu \Big(\sum_i \iota_i \pi_i\Big) \eta
        =  \sum_i(\psi\nu\iota_i)\circ( \pi_i\eta)
        =  \sum_i\tilde\circ(\psi\nu\iota_i\otimes_B \pi_i\eta)
        =  \tilde\circ\Big(\sum_i \psi\nu\iota_i\otimes_B \pi_i\eta\Big).
    \end{align*}

    For injectivity assume that $\tilde \circ (\sum_jf_j\otimes g_j)= 0$, for $f_j\in \cE(y,z)$ and $g_j\in \cE(x,y)$.
    Then

    \begin{align*}
        \sum_jf_j\tens_B g_j
        &= \sum_jf_j\tens_B \Big(\sum_ig_j \nu\iota_i\pi_i\eta \Big)
        = \sum_i\sum_j(f_jg_j \nu\iota_i)\tens_B \left(\pi_i\eta \right)
        = \sum_i0\tens_B \left(\pi_i\eta \right)
        = 0.
    \end{align*}
    
    \medbreak
    
    \ref{lem:item:wif3} Consider the diagram
    \begin{equation*}
    \begin{tikzcd}
        \cE(x,y)\otimes_A \cE(\tilde x,x)\dar["\pr \otimes \id"] \rar["\phi"]& \cE(\tilde x,y)\dar["\pr"]\\
        \cC(x,y)\otimes_A \cE(\tilde x,x) \rar["\psi"]& \cC(\tilde x,y)
    \end{tikzcd}
    \end{equation*}

    where the horizontal morphisms are induced by composition, and where $\pr$ denotes the projection. Since $\phi$ is an isomorphism and since $\pr$ is surjective, we get that $\psi$ is surjective.
    To check that $\psi$ is injective, suppose that $\psi(\sum_i \underline{g_i} \otimes_A f_i) =  \underline{\sum_i g_i f_i}= 0$, with $g_i\in \cE(x,y)$, and $f_i\in \cE(\tilde x,x)$. 
    As in the proof of \ref{lem:item:wif4} consider the diagram
    \begin{equation*}
    \begin{tikzcd}
        \tilde x \rar["\eta"{anchor=south}, bend left] & x^n \ar[l,"\nu"{anchor=north}, bend left]\rar["\pi_i", bend left] & x \ar[l,bend left, "\iota_i"],
    \end{tikzcd}
    \end{equation*}
    such  that $\nu \eta = \id_{\tilde x}$, and where $\pi_i$ is the projection to the $i$'th component, and $\iota_i$ is the inclusion of the $i$'th component. 
    Now
    \begin{align*}
        \sum_i \underline{g_i} \tens_A f_i
        &= \sum_i \underline{g_i} \tens_A f_i \sum_j \nu\iota_j\pi_j\eta
        = \sum_j\Big(\sum_i \underline{g_i} \tens_A (f_i \nu\iota_j\pi_j\eta)\Big)\\
        &= \sum_j\Big(\sum_i(\underline{g_if_i})\underline{ \nu\iota_j} \tens_A (\pi_j\eta)\Big)
        = 0.
    \end{align*}

    \ref{lem:item:wif7} 
    consider the following diagram
    \begin{equation*}
    \begin{tikzcd}
        y \rar["\eta"{anchor=south}, bend left] & x^n \ar[l,"\nu"{anchor=north}, bend left]\rar["\pi_i", bend left] & x \ar[l,bend left, "\iota_i"],
    \end{tikzcd}
    \end{equation*}
    such  that $\nu \eta = \id_y$, and where $\pi_i$ is the projection to the $i$'th component, and $\iota_i$ is the inclusion of the $i$'th component. 
    First we show injectivity. Let $\psi\in \cC(y,z)$, such that $\cC(x,\psi) = 0$.
    Assume for the sake of a contradiction that $\psi \neq 0$.

    \begin{align*}
        0\neq\psi\,\id_y
        &= \psi \nu\sum_i \iota_i\pi_i \eta
        = \sum_i\psi \nu \iota_i\pi_i \eta.
    \end{align*}

    Therefore there exists $k\in\{1,\dots,n\}$ such that $\psi \nu \iota_k\pi_k \eta\neq 0$.
    Thus $\cC(x,\psi)(\nu \iota_k ) = \psi \nu \iota_k \neq 0$, which is a contradiction.
    \medbreak

    For surjectivity let $\Phi\in \Hom_{\underline{A}}(\cC(x,y),\cC(x,z))$. 
    Let $\phi = \sum_i \Phi(\nu\iota_i)\pi_i\eta$, then we claim that $\Phi = \cC(x,\phi)$.
    For each $f\in\cC(x,y)$ calculate

    \begin{align*}
        \Phi(f)
        &= \Phi(\id_y f)
        = \Phi\Big(\sum_i \nu \iota_i\pi_i \eta f\Big)
        = \sum_i\Phi(\nu \iota_i)\pi_i \eta f
        = \phi f 
        = \cC(x,\phi)(f).
    \end{align*}

    \ref{lem:item:wif6} Since $x$ is maximal rigid, and $y,z$ are rigid, 
    \Cref{thm:ZZ_approx} says that there exist triangles
    \begin{equation*}
        x_1^{y} \rightarrow x_0^y \xrightarrow{\alpha} y \rightarrow \Sigma x_1^y
        \qquad\text{and}\qquad
        x_1^{z} \rightarrow x_0^z \xrightarrow{\beta} z \rightarrow \Sigma x_1^z,
    \end{equation*}

    with $x_i^z, x_i^y\in \add(x)$. Since $x$ is rigid this induces two exact sequences
    \begin{equation*}
        \cC(x,x_1^{y}) \rightarrow \cC(x,x_0^y) \xrightarrow{\alpha_*} \cC(x,y) \rightarrow 0
        \qquad\text{and}\qquad
        \cC(x, x_1^{z}) \rightarrow \cC(x,x_0^z) \xrightarrow{\beta_*} \cC(x,z) \rightarrow 0.
    \end{equation*}

    Now let $f\in \Hom_{\underline{A}}(\cC(x,y),\cC(x,z))$.
    Since $x_i^y,x_i^z\in\add(x)$, \ref{lem:item:wif7} says that
    there exist morphisms $g_i\in \cC(x_i^y,x_i^z)$ making the following diagram commute:

    \begin{equation*}
    \begin{tikzcd}
        \cC(x,x_1^{y})\rar\dar["{g_1}_{*}"] &
        \cC(x,x_0^{y})\rar["\alpha_*"]\dar["{g_0}_{*}"] &
        \cC(x,y)\rar\dar["f"] &
        0\\
        \cC(x,x_1^{z})\rar &
        \cC(x,x_0^{z})\rar["\beta_*"] &
        \cC(x,z)\rar &
        0.
    \end{tikzcd}
    \end{equation*}

    By the axioms of triangulated categories, there exists a morphism $g\in \cC(y,z)$,
    making $(g_1,g_0,g)$ a morphism of triangles
    \begin{equation*}
    \begin{tikzcd}
        x_1^{y}\rar\dar["{g_1}"] &
        x_0^{y}\rar["\alpha"]\dar["{g_0}"] &
        y\rar\dar["g"] &
        \Sigma x_1^y\dar["\Sigma g_1"]\\
        x_1^{z}\rar &
        x_0^{z}\rar["\beta"] &
        z\rar &
        \Sigma x_1^z.
    \end{tikzcd}
    \end{equation*}

    Thus we get a diagram
    \begin{equation*}
    \begin{tikzcd}
        \cC(x,x_1^{y})\rar\dar["{g_1}_{*}"] &
        \cC(x,x_0^{y})\rar["\alpha_*"]\dar["{g_0}_{*}"] &
        \cC(x,y)\rar\dar[shift right,"f"']\dar[shift left,"g_*"] &
        0\\
        \cC(x,x_1^{z})\rar &
        \cC(x,x_0^{z})\rar["\beta_*"] &
        \cC(x,z)\rar &
        0,
    \end{tikzcd}
    \end{equation*}

    with $f \alpha_* = \beta_* {g_0}_* =  g_* \alpha_*$. 
    Since $\alpha_*$ is an epimorphism this implies that $f=g_*$. \qedhere
\end{proof}

    Let $\D(-)=\Hom_k(-,k)$ denote $k$-duality.
\begin{lemma}
    \label{lem:hom_sym}
    Let $x\in\cC$ be a maximal rigid object,
    then $\Sigma^2x\cong x$ if and only if $\cC(x,x)$ is a self-injective algebra.
\end{lemma}
\begin{proof}
    Let $x'\in\cC$, and assume that $\Sigma^2 x\cong x$. Since $\cC$ is 2-Calabi--Yau there are isomorphisms
    \begin{equation}
        \label{eq:hom_sym:iso}
        \D\cC(x',x) \cong \D\cC(x',\Sigma^2 x) \cong \cC(x,x'),
    \end{equation}
    which are functorial in $x$. 
    This gives an isomorphism $\D\cC(x,x)\cong\cC(x,x)$ of right $\cC(x,x)$-modules making $\cC(x,x)$ self-injective.
    For the opposite implication one can do a argument similar to that of \cite[prop. 3.6]{iyama2013stable}.\qedhere
\end{proof}

Let $\Omega$ denote the syzygy, $\Omega^{-1}$ the cosyzygy  in $\cE$.
\begin{lemma}
    \label{lem:ext_is_C}
    Let $x,y\in \cE$, then \( \Ext^1_\cE(x,\Omega y) \cong \cC(x,y)\cong \Ext^1_\cE(\Omega^{-1}x,y).\)
\end{lemma}
\begin{proof}
    We will prove the first isomorphism. The second one follows by a dual argument.
    Recall that $\cE$ being Frobenius requires that there are enough projective objects. 
    Hence there exists a conflation

    \begin{equation*}
    \begin{tikzcd}
        0 \rar &
        \Omega y \rar["f"]&
        P \rar["g"]&
        y\rar &
        0
    \end{tikzcd}
    \end{equation*}

    with $P$ being projective and injective.
    Denote by $\pr$ the projection $\cE(x,y)\rightarrow \cC(x,y)$. 
    Since $\cC(x,y) = \cE(x,y)/\Ker(\pr)$, and since $\Ext^1(x,\Omega y)=\cE(x,y)/\Im(g_*)$, it is enough to show that $\Ker(\pr)=\Im(g_*)$.
    Since $\pr g_* = 0$ we have the inclusion $\Im(g_*)\subseteq \Ker(\pr)$.
    Thus it is enough to show that $\Ker(\pr)\subseteq\Im(g_*)$.
    \medbreak

    Let $h\in\Ker(\pr)$, then $h$ factors through a projective object $\tilde P$,
    say $h = p'p$ where $p: x\rightarrow\tilde P$, $p':\tilde P\rightarrow y$.
    Using that $g$ is an epimorphism,
    there exists a morphism $w:\tilde P \rightarrow P$ making the following diagram commute:

    \begin{equation*}
    \begin{tikzcd}
        \tilde P\rar[dashed, "w"]\drar["p'"]&P\dar[twoheadrightarrow, "g"]\\
         x\uar["p"]\rar["h"]&y.
    \end{tikzcd}
    \end{equation*}
        
    In other words, $h =g_*(wp)$, thus $h\in \Im g_*$. \qedhere

\end{proof}

\begin{lemma}
    \label{lem:aug_proj_res_C}
    Let $x\in \cE$ be good, such that $\Sigma^2 x \cong x$. Let $A=\cE(x,x)$
    Then there are conflations in $\cE$:
    \begin{equation}
    \begin{gathered}
        \begin{tikzcd}
            0 \rar& 
            \Omega x \rar["g"]& 
            p_1 \rar["f"]& 
            x \rar& 
            0,
        \end{tikzcd}\\[-5pt]
        \begin{tikzcd}
            0 \rar& 
            x\oplus p_3\rar["g'"]& 
            p_2 \rar["f'"]& 
            \Omega x \rar& 
            0,
        \end{tikzcd}
    \end{gathered}
    \end{equation}

    with $p_i$ being projective objects.
    From these conflations follows augmented projective resolutions 
    \begin{equation}
        \label{eq:aug_proj1}
        0\longrightarrow
         \cE(x, x \oplus p_3) \longrightarrow
         \cE(x, p_2) \longrightarrow
         \cE(x, p_1) \longrightarrow
         \cE(x,x)\longrightarrow
         \cC(x,x),
    \end{equation}
    \begin{equation}
        \label{eq:aug_proj2}
        0\longrightarrow
         \cE(x,x)\longrightarrow
         \cE(p_1,x) \longrightarrow
         \cE(p_2,x) \longrightarrow
         \cE(x \oplus p_3,x) \longrightarrow
         \cC(x,x).
    \end{equation}
    \medbreak
    
    Here the first one is a projective resolution of right $A$-modules,
    and the second is a projective resolution of left $A$-modules.
\end{lemma}
\begin{proof}
    Let $0\to\Omega x\to p_1\to x\to 0$ and $0\to\Omega^2 x\to p_2'\to \Omega x\to 0$ be conflations in $\cE$ with $p_1,p_2'$ projective. 
    Since $\Omega^2 x\cong \Sigma^{-2} x\cong x$ in $\cC$, \Cref{lem:wif}\ref{lem:item:wif2} says that there are projective objects
    $p_3,q\in \cE$ such that $\Omega^2 x \oplus q \cong x\oplus p_3$. 
    Hence there is an exact sequence spliced from two conflations

    \begin{equation*}
    \begin{tikzcd}
        0\rar& x \oplus p_3 \rar["g'"]& p_2 \ar["f'"',dr,twoheadrightarrow]\ar[rr]&&p_1\rar["f"]&x\rar&0,\\
             &&& \Omega x\ar["g"',ur,rightarrowtail] &&&
    \end{tikzcd}
    \end{equation*}
    
    with $p_i\in\cE$ being projective objects where $p_2 = p_2' \oplus q$.
    Applying the functor $\cE(x,-)$ gives the exact sequence

    \begin{equation*}
    \begin{tikzcd}[column sep=1.7em]
          0 \ar[r]
        & \cE(x,x\oplus p_3) \ar[r,"g'_*"]
        & \cE(x,p_2) \ar[dr,twoheadrightarrow, "f'_*"']\ar[rr]
        && \cE(x,p_1) \ar[r, "f_*"]
        & \cE(x,x).\\
        &&& \cE(x,\Omega x) \ar[ur,rightarrowtail,"g_*"']&&
    \end{tikzcd}
    \end{equation*}

    That $f'_*$ is surjective follows directly from $x$ being rigid and therefore $\Ext^1(x,x\oplus p_3)=0$, see \Cref{lem:ext_is_C}. 
    Since $\Ext^1(x,p_1)=0$ it follows directly from \Cref{lem:ext_is_C} that this is a projective resolution of $\cC(x,x)$ over right $\cE(x,x)$-modules. 
    This shows that we have the resolution from \cref{eq:aug_proj1}. 
    The method for finding the resolution from \cref{eq:aug_proj2} is similar.\qedhere
\end{proof}

\begin{lemma}
    \label{lem:tor2_is_0}
    Let $x\in\cE$ be good and let $y\in\cE$. Define $A\coloneqq \cE(x,x)$, and $\underline{A}\coloneqq \cC(x,x)$. If $x$ is rigid in $\cC$ and  $\Sigma^2 x \cong x$ then
    \( \Tor_2^A(\cC(x,y), {}_{A}\underline{A}) \cong 0. \)
\end{lemma}
\begin{proof}
    By \Cref{lem:aug_proj_res_C} there is an augmented projective resolution

    \begin{equation*}
    \begin{tikzcd}[column sep=1.5em]
        0\rar&
        \cE(x,x)\rar&
        \cE(p_1,x)\rar&
        \cE(p_2,x)\rar&
        \cE(x\oplus p_3,x)\rar&
        \cC(x,x)\rar&
        0,
    \end{tikzcd}
    \end{equation*}

    of $\tensor[_A]{\uA}{}$ over left $A$-modules, with $p_i$ being projective objects.
    Using \Cref{lem:wif}\ref{lem:item:wif3} gives that 
    \begin{align*} 
       \cC(x,y)\ltensor_{A}\tensor[_A]{\underline{A}}{}
        &\cong \cC(x,y)\tens_{A}\big(
            \begin{tikzcd}[column sep=1.5em, ampersand replacement=\&]
                0\rar\&
                \cE(x,x)\rar\&
                \cE(p_1,x)\rar\&
                \cE(p_2,x)\rar\&
                \cE(x\oplus p_3,x)\rar\&
                0
            \end{tikzcd}
        \big)\\
       &\cong \begin{tikzcd}[column sep=1.5em, ampersand replacement=\&]
        0\rar\&
        \cC(x,y)\rar\&
        \cC(p_1,y)\rar\&
        \cC(p_2,y)\rar\&
        \cC(x\oplus p_3,y)\rar\&
        0
    \end{tikzcd}\\
    &\cong \begin{tikzcd}[column sep=1.5em, ampersand replacement=\&]
        0\rar\&
        \cC(x,y)\rar\&
        0\rar\&
        0\rar\&
        \cC(x,y)\rar\&
        0.
    \end{tikzcd}
    \end{align*}
    
    Hence $\Tor_2^A(\cC(x,y), {}_{A}\underline{A}) \cong 0$.
\end{proof}

\begin{lemma}
    \label{lem:proj_ideal_cap}
    Let $x\in\cE$ be good and assume that $x$ is rigid in $\cC$ and satisfies $\Sigma^2 x \cong x$. 
    Define $A\coloneqq \cE(x,x)$, and $\underline{A}\coloneqq \cC(x,x)$. Let $\fa$ be the ideal in $A$ of morphisms factoring through a projective object.
    Let $h:M\rightarrow Q$ be a morphism of right $A$-modules with $Q$ being a projective object, such that $\Coker h \cong\cC(x,y)$ for some $y\in\cC$.
    Then $M\fa \cap \Ker h = (\Ker h)\fa$.
\end{lemma}
\begin{proof}
    Using the $\Tor_2$ vanishing of \Cref{lem:tor2_is_0}, this can be proved by the same argument as \cite[lem. 3.7]{august2020tilting}.
\end{proof}

\section{Derived equivalences}
In this section, \Cref{set:EC} together with the following setup will be assumed.
\begin{setup}
    Let $l,m\in\cC$ be good maximal rigid objects, such that $\Sigma^2 l \cong l$, and $\Sigma^2 m \cong m$ in $\cC$. 
    Let $A = \cE(l,l)$, $\underline{A} = \cC(l,l)$, $B = \cE(m,m)$, $\underline{B} = \cC(m,m)$ and ${}_B T_A = \cE(l,m)$. 
    The following construction of a two-sided tilting complex is inspired by \cite[p. 5123]{mizuno2019derived} and \cite[thm. 1.1]{august2020tilting}. 
    \[{}_\uB \cT_\uA =  \left(\tensor[_\uB]{\underline{B}}{_B} \ltensor_B \tensor[_B]{T}{_A} \ltensor_A \tensor[_A]{\uA}{_\uA}\right)_{\subseteq 1},\]

    where $\subseteq 1$ refers to taking a soft truncation, keeping the homological degrees $\leq 1$.
\end{setup}
The main goal of this section is to show that 
${}_{\underline B}\cT_{\underline{A}}$ is a 2-sided tilting complex, making $\underline{A}$ and $\underline{B}$ derived equivalent (\Cref{cor:Ptilting}).

\begin{lemma}[{\cite[prop 5.1]{jorgensen2019green}}]
    \label{lem:t_tilting}
    ${}_B T_A$ is a tilting complex viewed as a $B\otimes A^{\op}$ complex.
\end{lemma}

Applying \Cref{lem:aug_proj_res_C} to our current setup there are conflations in $\cE$:
    \begin{equation}
    \label{eq:m_conflations}
        \begin{gathered}
            \begin{tikzcd}
                0 \rar& 
                \Omega m \rar["g"]& 
                p_1 \rar["f"]& 
                m \rar& 
                0,& 
            \end{tikzcd}\\[-5pt]
            \begin{tikzcd}
                0 \rar& 
                m\oplus p_3\rar["g'"]& 
                p_2 \rar["f'"]& 
                \Omega m \rar& 
                0& 
            \end{tikzcd}
        \end{gathered}
    \end{equation}

where $p_i\in \cE$ are projective objects, giving a projective resolution

    \begin{equation}
        \label{eq:proj_res_uB}
        Q_B: \quad
        0\longrightarrow
         \cE(m, m \oplus p_3) \longrightarrow
         \cE(m, p_2) \longrightarrow
         \cE(m, p_1) \longrightarrow
         \cE(m,m)
    \end{equation}

    of $\underline{B}_B$.

\begin{lemma}
    \label{lem:B_tens_T_homology}
    In $\pD(A^{\op})$ the object $\tensor{\underline{B}}{_B} \ltensor_B \tensor[_B]{T}{_A}$ is quasi-isomorphic to the complex 

    \begin{equation}
        \label{eq:exact_seq_m}
    \begin{tikzcd}[column sep=2.6em]
          0 \ar[r]
        & \cE(l,m\oplus p_3) \ar[r,"g'_*"]
        & \cE(l,p_2) \ar[r,"(gf')_*"]
        & \cE(l,p_1) \ar[r, "f_*"]
        & \cE(l,m),
    \end{tikzcd}
    \end{equation}
     with homology
     \[
         H_i(\tensor{\underline{B}}{_B} \ltensor_B \tensor[_B]{T}{_A}\!)=
           \begin{cases}
               \cC(l,m)& i=0,\\
               \cC(l,\Omega m)& i=1,\\
               0 & \text{otherwise.}
           \end{cases}
     \]
\end{lemma}
\begin{proof}
    Using the projective resolution $Q_B$ of $\underline{B}_B$ from \cref{eq:proj_res_uB}, calculate
    \begin{align*}
        \tensor{\underline{B}}{_B} &\ltensor_B \tensor[_B]{T}{_A}
        \cong \tensor{Q}{_B} \ltensor_{B} \tensor[_B]{T}{_A} \\
        &\cong 
            \begin{tikzcd}[column sep=.9em, ampersand replacement=\&]
                  0 \ar[r]
                \& \cE(m,m\oplus p_3)\,\underset{B}{\tens}\, \tensor[_B]{T}{_A} \ar[r]
                \& \cE(m,p_2)\, \underset{B}{\tens}\, \tensor[_B]{T}{_A} \ar[r]
                \& \cE(m,p_1)\,\underset{B}{\tens}\, \tensor[_B]{T}{_A} \ar[r]
                \& \cE(m,m)\,\underset{B}{\tens}\,   \tensor[_B]{T}{_A}
            \end{tikzcd}\\
        &\cong 
            \begin{tikzcd}[column sep=.9em, ampersand replacement=\&]
                  0 \ar[r]
                \& \cE(l,m\oplus p_3)\ar[r]
                \& \cE(l,p_2)\ar[r]
                \& \cE(l,p_1)\ar[r]
                \& \cE(l,m),
            \end{tikzcd}
    \end{align*}

where the last isomorphism follows from \Cref{lem:wif}\ref{lem:item:wif4}.

\medbreak
This complex can also be seen as the result of applying the functor $\cE(l,-)$ to the concatenation of the conflations from \cref{eq:m_conflations}:

\begin{equation*}
    \begin{tikzcd}[column sep=1.7em]
          0 \ar[r]
        & \cE(l,m\oplus p_3) \ar[r,"g'_*"]
        & \cE(l,p_2) \ar[dr, "f'_*"']\ar[rr, dashed, "(gf')_*"]
        && \cE(l,p_1) \ar[r, "f_*"]
        & \cE(l,m).\\
        &&& \cE(l,\Omega m) \ar[ur,rightarrowtail,"g_*"']&&
    \end{tikzcd}
\end{equation*}

Thus $H_i(\tensor{\underline{B}}{_B} \ltensor_B \tensor[_B]{T}{_A}\!) = 0$ for $i\neq 0,1$.
Since $\Ext^1_{\cE}(l,p_1) \cong 0$, it follows from \Cref{lem:ext_is_C} together with the long exact $\Ext$ sequence of the conflation
\[
    0           \longrightarrow
    \Omega m    \longrightarrow 
    p_1         \longrightarrow 
    m           \longrightarrow
    0,
\]
that $H_0(\tensor{\underline{B}}{_B} \ltensor_B \tensor[_B]{T}{_A}\!) \cong \cC(l,m)$. 

\indent To calculate $H_1$, notice that $\Ker(f_*) \cong \cE(l,\Omega m)$, and $\Im((gf')_*) \cong \Im(f'_*)$, due to $g_*$ being injective. Therefore $H_1(\underline{B}_B \ltensor_B {}_B T_A) \cong \cE(l,\Omega m)/\Im(f'_*)$.
Using that $\Ext^1_{\cE}(l,p_2) \cong 0$, the long exact $\Ext$ sequence of the conflation
\[
    0           \longrightarrow
    m \oplus p_3   \longrightarrow 
    p_2\longrightarrow
    \Omega m           \longrightarrow
    0
\]

gives that $ H_1(\underline{B}_B \ltensor_B {}_B T_A) \cong \Ext^1(l,m\oplus p_3) \cong \cC(l,\Omega^{-1} m)$, with the last isomorphism coming from \Cref{lem:ext_is_C}. However, $\Sigma^{2} m\cong m$ means $\Omega^{-2} m \cong m$ so $\Omega^{-1} m \cong \Omega m$.
\qedhere

\end{proof}

\begin{corollary}
    \label{cor:b_tens_T_1}
    There is a quasi-isomorphism in $\pD(A^{\op})$ from $\tensor{\underline{B}}{_B}\ltensor_B \tensor[_B]{T}{_A}$ to the complex
    \begin{equation}
        \label{eq:P-quasi-iso}
        \begin{tikzcd}[column sep=1.7em]
            \cE(l,p_1) / \Im((gf')_*) \ar[r, "f_*"]
            & \cE(l,m). 
        \end{tikzcd}
    \end{equation}
\end{corollary}
\begin{proof}
    This follows directly from \Cref{lem:B_tens_T_homology} by using soft truncation.
\end{proof}

The next goal is to obtain an alternative description of  $\tensor{\underline{B}}{_B}\ltensor_B \tensor[_B]{T}{_A}$.
Since $l$ is good and maximal rigid, and since $m$ is rigid, there exists a conflation
\[
    \begin{tikzcd}[column sep=1.7em]
        0 \rar &
        l_1 \rar["\phi_1"] &
        l_0 \rar["\phi_0"] &
        m \rar &
        0,
    \end{tikzcd}
\]

where $l_i \in \add(l)$ (see \Cref{thm:ZZ_approx} and \Cref{lem:wif}\ref{lem:item:wif1}). For the rest of this section denote by $P_\uA$ the complex
of $\underline{A}^{\op}$-modules
\[
    P_\uA:\begin{tikzcd}[column sep=2em]
        \cC(l,l_1) \rar["{\phi_1}_*"] &
        \cC(l,l_0),
    \end{tikzcd}
\]

concentrated in degrees 0,1. The claim is that $\tensor{\underline{B}}{_B}\ltensor_B \tensor[_B]{T}{_A} \cong P_A$ in $\pD(A^{\op})$.
To show this is the case, we will find a complex of projective objects which is quasi-isomorphic to the complex from \Cref{cor:b_tens_T_1},
and show that it is also quasi-isomorphic to $P_A$.
\medbreak

Firstly, notice that since $p_1$ is projective there is a push-out diagram
\begin{equation*}
\begin{tikzcd}
    0\rar&
    \Omega m \dar["\gamma_1"]\ar[PO]\rar["g"]&
    p_1\dar["\gamma_0"] \rar["f"]&
    m \dar[equal]\rar& 0\\
    0\rar&  l_1 \rar["\phi_1"]& l_0 \rar["\phi_0"]& m \rar& 0,
\end{tikzcd}
\end{equation*}

which gives a conflation (see \cite[lem 2.12]{buhler2010exact})
\begin{equation}
    \label{eq:pushout_exact_seq}
    \begin{tikzcd}[ampersand replacement=\&, column sep=3.5em]
    0 \rar \&
    \Omega m \rar["{\begin{psmallmatrix} g\\\gamma_1 \end{psmallmatrix}}"] \&
    p_1\oplus l_1 \rar["{\begin{psmallmatrix} -\gamma_0& \phi_1 \end{psmallmatrix}}"] \&
    l_0 \rar \&
    0.
\end{tikzcd}
\end{equation}

\begin{lemma}
    \label{lem:projres_cP}
    There is a complex $\cP_A$ in $\pD(A^{\op})$ of projective objects, given by
    \begin{equation*}
     \begin{tikzcd}[ampersand replacement=\&, column sep=3em]
        0 \rar\& 
        \cE(l,l_1) \rar["{\begin{psmallmatrix} \phi_1\\ 0 \end{psmallmatrix}_{\hspace{-2pt}*}}"]\& 
        \cE(l, l_0 \oplus p_3) \rar["\psi_*"]\& 
        \cE(l,p_2) \rar["{\begin{psmallmatrix} gf'\\ \gamma_1f' \end{psmallmatrix}_{\hspace{-2pt}*}}"]\& 
        \cE(l, p_1 \oplus l_1 ) \rar["{\begin{psmallmatrix} {-\gamma_0} &\hspace{0pt} \phi_1 \end{psmallmatrix}_{\hspace{-1pt}*}}"]\&
        \cE(l, l_0),
    \end{tikzcd}
    \end{equation*}

    that is quasi-isomorphic to the complex from \cref{eq:P-quasi-iso}.

\end{lemma}
\begin{proof}
    Firstly, by \Cref{lem:ext_is_C} there is an exact sequence
    \begin{equation*}
        \label{eq:E_on_appr}
    \begin{tikzcd}
        0\rar&
        \cE(l,l_1) \rar["{\phi_1}_*"]& 
        \cE(l, l_0) \rar["{\phi_0}_*"]&
        \cE(l,m)\rar&0.
    \end{tikzcd}
    \end{equation*}

    From this follows another exact sequence: 

    \begin{equation}
        \label{eq:E_on_appr_w_p}
        \begin{tikzcd}[ampersand replacement=\&,column sep=3em]
        0\rar\&
        \cE(l,l_1) \rar["{\begin{psmallmatrix} \phi_1\\0 \end{psmallmatrix}}_{\hspace{-2pt}*}"]\& 
        \cE(l, l_0\oplus p_3) \rar["{ \begin{psmallmatrix} \phi_0&0\\0&\id \end{psmallmatrix}}_{\hspace{-2pt}*}"]\&
        \cE(l,m\oplus p_3)\rar\&0.
    \end{tikzcd}
    \end{equation}

    Now let
    \begin{equation*}
        \psi = g'\begin{pmatrix}
            \phi_0 & 0\\ 0&\id
        \end{pmatrix} : l_0 \oplus p_3 \xrightarrow{\begin{psmallmatrix} \phi_0 & 0\\ 0&\id \end{psmallmatrix}} m\oplus p_3 \xrightarrow{\ g'\ } p_2. 
    \end{equation*}

    By concatenating the sequences from \cref{eq:exact_seq_m,eq:E_on_appr_w_p}, followed by a soft truncation, the following exact sequence is obtained:
    \begin{equation*}
        \begin{tikzcd}[ampersand replacement=\&,column sep=2em]
        0 \rar\& 
        \cE(l,l_1) \rar["{\begin{psmallmatrix} \phi_1\\ 0 \end{psmallmatrix}_{\hspace{-2pt}*}}"]\& 
        \cE(l, l_0 \oplus p_3) \rar["\psi_*"]\& 
        \cE(l,p_2) \rar["(gf')_*"]\& 
        \cE(l, p_1)\rar["\pr"]\& 
        \cE(l,p_1) / \Im(gf')_* \rar\&0.
    \end{tikzcd}
    \end{equation*}

    This exact sequence is an augmented projective resolution of $\cE(l,p_1) / \Im(gf')_*$.
    Consider the following double complex $C$:

    \begin{equation*}
        \begin{tikzcd}[ampersand replacement=\&,column sep=2em]
        0 \rar\dar\& 
        \cE(l,l_1) \rar["{\begin{psmallmatrix} \phi_1\\ 0 \end{psmallmatrix}_{\hspace{-2pt}*}}"]\dar\& 
        \cE(l, l_0 \oplus p_3) \rar["\psi_*"]\dar\& 
        \cE(l,p_2) \dar["(\gamma_1 f')_*"]\rar["(gf')_*"]\& 
        \cE(l, p_1) \dar["{\gamma_0}_*"]\rar["\pr"]\& 
        \cE(l,p_1) / \Im(gf')_* \dar["f_*"]\rar\&0\dar\\
        0 \rar\& 
        0 \rar\& 
        0 \rar\& 
        \cE(l,l_1) \rar["{\phi_1}_*"]\& 
        \cE(l, l_0) \rar["{\phi_0}_*"]\&
        \cE(l,m)\rar\&0,
    \end{tikzcd}
    \end{equation*}
    where each row is an augmented projective resolution of $A^{\op}$-modules.
    This means that each row is exact and thus the total complex $\Tot^{\oplus}(C)$ is exact, see \cite[lem. 2.7.3]{weibel1994introduction}.
    Next consider the following morphism of complexes:
    \begin{equation*}
     \begin{tikzcd}[ampersand replacement=\&, column sep=2.3em]
        0 \rar\& 
        \cE(l,l_1) \rar["{\begin{psmallmatrix} \phi_1\\ 0 \end{psmallmatrix}_{\hspace{-2pt}*}}"]\& 
        \cE(l, l_0 \oplus p_3) \rar["\psi_*"]\& 
        \cE(l,p_2) \rar["{\begin{psmallmatrix} gf'\\ \gamma_1f' \end{psmallmatrix}_{\hspace{-2pt}*}}"]\& 
        \cE(l, p_1 \oplus l_1 ) \rar["{\begin{psmallmatrix} {-\gamma_0} &\hspace{-0pt} \phi_1 \end{psmallmatrix}_{\hspace{-1pt}*}}"]\dar["\pr\circ{\begin{psmallmatrix}\id_{p_1}&0 \end{psmallmatrix}_{\hspace{-1pt}*} }"]\&
        \cE(l, l_0)\dar["{\phi_0}_*"]\\
        \&\&\&\&
        \cE(l,p_1) / \Im(gf')_*\rar["f_*"']\&
        \cE(l,m).
    \end{tikzcd}
    \end{equation*}

    By \cite[ex. 1.2.8]{weibel1994introduction}, the mapping cone of this morphism is $\Tot^{\oplus}(C)$ up to a sign.
    Since $\Tot^{\oplus}(C)$ is acyclic, this means that the morphism is a quasi-isomorphism.
    We now have a complex of projective objects $\cP_A$: 
    \begin{equation*}
     \begin{tikzcd}[ampersand replacement=\&, column sep=3em]
        0 \rar\& 
        \cE(l,l_1) \rar["{\begin{psmallmatrix} \phi_1\\ 0 \end{psmallmatrix}_{\hspace{-2pt}*}}"]\& 
        \cE(l, l_0 \oplus p_3) \rar["\psi_*"]\& 
        \cE(l,p_2) \rar["{\begin{psmallmatrix} gf'\\ \gamma_1f' \end{psmallmatrix}_{\hspace{-2pt}*}}"]\& 
        \cE(l, p_1 \oplus l_1 ) \rar["{\begin{psmallmatrix} {-\gamma_0} &\hspace{0pt} \phi_1 \end{psmallmatrix}_{\hspace{-1pt}*}}"]\&
        \cE(l, l_0),
    \end{tikzcd}
    \end{equation*}
    which is quasi-isomorphic to the complex from \cref{eq:P-quasi-iso}.\qedhere
\end{proof}

For the rest of this section we will use the complex $\cP_A$ from the previous lemma.

\begin{lemma}
    \label{lem:dwp}
    The degreewise projection $\Phi: \cP_A\rightarrow P_A$
\begin{equation*}
 \begin{tikzcd}[ampersand replacement=\&, column sep=3em]
    0 \rar\& 
    \cE(l,l_1) \rar["{\begin{psmallmatrix} \phi_1\\ 0 \end{psmallmatrix}_{\hspace{-2pt}*}}"]\& 
    \cE(l, l_0 \oplus p_3) \rar["\psi_*"]\& 
    \cE(l,p_2) \rar["{\begin{psmallmatrix} gf'\\ \gamma_1f' \end{psmallmatrix}_{\hspace{-2pt}*}}"]\& 
    \cE(l, p_1 \oplus l_1 ) \rar["{\begin{psmallmatrix} {-\gamma_0} &\hspace{-0pt} \phi_1 \end{psmallmatrix}_{\hspace{-1pt}*}}"]\dar["\pr\circ{\begin{psmallmatrix}0&\id_{l_1} \end{psmallmatrix}_{\hspace{-1pt}*} }"]\&
    \cE(l, l_0)\dar["\pr"]\\
    \&\&\&\&
    \cC(l,l_1)\rar["\phi_{1*}"']\&
    \cC(l,l_0)
\end{tikzcd}
\end{equation*}

is a quasi-isomorphism.

\end{lemma}
\begin{proof}

By \Cref{lem:projres_cP} the top row $\cP_A$ is quasi-isomorphic to $\tensor{\underline{B}}{_B}\ltensor_B \tensor[_B]{T}{_A}$ and hence 
 is exact in all degrees other than 0,1 by \Cref{lem:B_tens_T_homology}.
Thus to prove the lemma, it is enough to show that $H_i(\Phi)$ is an isomorphism for $i=0,1$. 
First we check that the homology of the two complexes agree, after which
it is enough to check that $H_i(\Phi)$ is surjective because the homology spaces are finite dimensional over $k$.
\medbreak
    
Since there is a triangle $\Omega m\to l_1 \to l_0 \to m$, and since $l$ is rigid with $\Omega l\cong \Sigma^{-1} l\cong \Sigma l$, 
there is an exact sequence

\begin{equation*}
\begin{tikzcd}
    0 \rar&
    \cC(l,\Omega m) \rar&
    \cC(l,l_1) \rar&
    \cC(l, l_0) \rar&
    \cC(l, m) \rar&
    0.
\end{tikzcd}
\end{equation*}

In combination with \Cref{cor:b_tens_T_1} and \Cref{lem:B_tens_T_homology,lem:projres_cP}
this gives that $H_i(P_A) \cong H_i(\cP_A)$, for every $i\in\ZZ$.

\medbreak
Next we show that $H_i(\Phi)$ is surjective. For $i=0$ this is straightforward.
For $i=1$ consider the diagram

\begin{equation*}
 \begin{tikzcd}[ampersand replacement=\&, column sep=3em]
     \cE(l, p_1 \oplus l_1 ) \rar["{\begin{psmallmatrix} {-\gamma_0} &\phi_1 \end{psmallmatrix}_{\hspace{-2pt}*}}"]\dar["{\pr \begin{psmallmatrix}0&\id_{l_1}\end{psmallmatrix} }_*"]\&
    \cE(l, l_0)\dar["\pr"]\\
    \cC(l,l_1)\rar["{\phi_1}_*"]\&
    \cC(l,l_0).
\end{tikzcd}
\end{equation*}

Let $\alpha\in \cE(l,l_1)$ such that the projection $\underline{\alpha}\in \cC(l,l_1)$ lies in $\Ker{\phi_1}_*$.
Thus $\underline{\phi_1\alpha}=0$, which means that $\phi_1\alpha$ factors through a projective object $q\in\cE$, say by $\phi_1\alpha = \rho'\rho$,
where $\rho:l\rightarrow q$ and $\rho':q\rightarrow l_0$.
By using \cref{eq:pushout_exact_seq} there is a commutative diagram

\begin{equation*}
    \begin{tikzcd}[ampersand replacement=\&]
        \&q\drar["\rho'"]\rar["{\begin{pmatrix} a_1 \\ a_2 \end{pmatrix} }", dashed]
    \&p_1\oplus l_1 \dar["{\begin{pmatrix} -\gamma_0 &\hspace{-5pt}\phi_1 \end{pmatrix} }",twoheadrightarrow]\\
    l\rar["\alpha"]\urar["\rho"]\&l_1\rar["\phi_1"]\&l_0,
\end{tikzcd}
\end{equation*}

where $a_1,a_2$ exist by using the fact that $q$ is projective.
Now calculate
\begin{equation*}
    \phi_1\alpha = \begin{pmatrix} -\gamma_0 & \phi_1 \end{pmatrix}\begin{pmatrix}
        a_1 \\ a_2
    \end{pmatrix} \rho
    = \phi_1 a_2 \rho - \gamma_0 a_1 \rho.
\end{equation*}

Therefore
\begin{equation*}
    0 = \phi_1 (a_2\rho - \alpha) - \gamma_0a_1\rho
    =\begin{pmatrix} -\gamma_0 & \phi_1 \end{pmatrix} \begin{pmatrix} a_1\rho \\ a_2\rho-\alpha \end{pmatrix}
    =\begin{pmatrix} -\gamma_0 & \phi_1 \end{pmatrix}_*\begin{pmatrix} a_1\rho \\ a_2\rho-\alpha \end{pmatrix}.
\end{equation*}

This means that 
$\begin{pmatrix}a_1\rho\\ a_2\rho-\alpha\end{pmatrix} \in\Ker \begin{pmatrix} -\gamma_0 & \phi_1 \end{pmatrix}_*$,
making $\begin{pmatrix}a_1\rho\\ a_2\rho-\alpha \end{pmatrix}+ \Im \begin{pmatrix} gf'\\ \gamma_1f' \end{pmatrix}_{\hspace{-2pt}*}$ an element of $H_1(\cP_A)$. Now
\begin{equation*}
    H_1(\Phi)\Bigg(-\begin{pmatrix} a_1\rho \\ a_2\rho-\alpha \end{pmatrix}
    + \Im \begin{pmatrix} gf'\\ \gamma_1f' \end{pmatrix}_{\hspace{-2pt}*}\Bigg) = \underline\alpha-\underline{a_2\rho }  = \underline \alpha,
\end{equation*}
which means that $H_1(\Phi)$ is surjective, and therefore $\Phi$ is surjective on homology.\qedhere

\end{proof}

\begin{corollary} There are two isomorphisms 
\label{cor:oneside_theorem}
    \begin{enumerate}[label=(\alph*)]
        \item $\cT_\uA =  \left(\tensor{\underline{B}}{_B} \ltensor_B \tensor[_B]{T}{_A} \ltensor_A \tensor[_A]{\uA}{_\uA}  \right)_{\subseteq 1} \cong P_\uA$ in $\pD(\underline{A}^{\op})$. \label{cor:sim:1}
        \item $\tensor{\cT}{_\uA} \ltensor_{\underline{A}} \tensor[_\uA]{\uA}{_A}  \cong  \tensor{\underline{B}}{_B} \ltensor_B \tensor[_B]{T}{_A}$ in $\pD(A^{\op})$.\label{cor:sim:2}
    \end{enumerate}
\end{corollary}

\begin{proof}
\ref{cor:sim:1}
\Cref{cor:b_tens_T_1} and \Cref{lem:projres_cP} imply $\cP_A\cong \tensor{\underline{B}}{_B}\ltensor_{B} \tensor[_B]{T}{_A}$. \Cref{lem:wif}\ref{lem:item:wif3} implies that $\tensor{\cP}{_A} \tens_A \tensor[_A]{\uA}{_\uA} \cong P_\uA\oplus P_\uA[3]$. 
Thus
\[
    \tensor{\underline{B}}{_B} \ltensor_B \tensor[_B]{T}{_A} \ltensor_A \tensor[_A]{\uA}{_\uA}
    \cong \tensor{\cP}{_A} \tens_A \tensor[_A]{\uA}{_\uA}
    \cong \tensor{P}{_\uA}\oplus \tensor{P}{_\uA}[3],
\]
giving that
\[
    \cT_\uA \cong (P_\uA \oplus P_\uA[3])_{\subseteq 1} \cong P_\uA.
\]

\ref{cor:sim:2}  It follows from \ref{cor:sim:1} combined with \Cref{cor:b_tens_T_1} and \Cref{lem:projres_cP,lem:dwp} that
\[
    \cT_\uA \ltensor_{\underline{A}} {}_\uA\underline{A}_A
    \cong P_\uA\ltensor_{\underline{A}} {}_\uA\underline{A}_A
    \cong P_A
    \cong \cP_A
    \cong \tensor{\underline{B}}{_B} \ltensor_B \tensor[_B]{T}{_A}.\qedhere
\]
    
\end{proof}

We now have a one-sided isomorphism $\cT_\uA \ltensor_{\underline{A}} {}_\uA\underline{A}_A  \cong  \tensor{\underline{B}}{_B} \ltensor_B\tensor[_B]{T}{_A} $ in $\pD(A^{\op})$. 
Next we see that this isomorphism can be lifted to a two-sided isomorphism, i.e. an isomorphism in $\pD(\underline{B}\otimes A^{\op})$.

\begin{theorem} 
    \label{thm:cTT-iso}
    In the derived category $\pD(\underline{B}\otimes A^{\op})$ (resp. $\pD(B\otimes \underline{A}^{\op})$), there is an isomorphism
    \[  
        \tensor[_\uB]{\!\cT}{_\uA} \ltensor_{\underline{A}} \tensor[_\uA]{\uA}{_A} \cong  \tensor[_\uB]{\underline{B}}{_{B}} \ltensor_B 
        \tensor[_B]{T}{_A}
        \qquad\left(\text{resp.}\hspace{.4em} \tensor[_B]{\underline{B}}{_\uB}\ltensor_\uB \,\tensor[_\uB]{\!\cT}{_\uA}  
        \cong \tensor[_B]{T}{_A} \ltensor_A \tensor[_A]{\underline{A}}{_\uA}\right). 
    \]
\end{theorem}
\begin{proof}
    
    We will prove the first isomorphism, the second one can be done in a symmetric fashion.
    Let $\tensor[_\uB]{Q}{_A} \simto \tensor[_\uB]{\underline{B}}{_B} \ltensor_B \tensor[_B]{T}{_A}$ be a projective resolution over $\underline{B}\otimes A^{\op}$.
    Denote by $\alpha:\tensor[_\uB]{Q}{_A}\rightarrow \tensor[_\uB]{Q}{_A}\otimes_A \!\tensor[_A]{\underline{A}}{_A}$ the morphism defined by $q\mapsto q\otimes 1$.
    Consider the composition of morphisms 
\[
    \pr\circ \alpha: \tensor[_\uB]{Q}{_A}
    \xrightarrow{\alpha} \tensor[_\uB]{Q}{_A} \tens_A \tensor[_A]{\underline{A}}{_A}
    \xrightarrow{\ \pr\ } (\tensor[_\uB]{Q}{_A} \tens_A \tensor[_A]{\underline{A}}{_A})_{\subseteq 1}.
\]

By construction this is isomorphic to

\begin{equation*}
    \tensor[_\uB]{\underline{B}}{_B}\ltensor_{B} \tensor[_B]{T}{_A} \longrightarrow
    \tensor[_\uB]{\underline{B}}{_B}\ltensor_{B} \tensor[_B]{T}{_A} \ltensor_A \tensor[_A]{\underline{A}}{_A} \longrightarrow
    \left(\tensor[_\uB]{\uB}{_B}\ltensor_{B} \tensor[_B]{T}{_A} \ltensor_A \tensor[_A]{\uA}{_A}\right)_{\subseteq 1}
    \cong \,\tensor[_\uB]{\!\cT}{_{\underline{A}}}\ltensor_{\underline{A}} \tensor[_\uA]{\uA}{_A}.
\end{equation*}

Writing out the composition gives the following, which we will show to be a quasi-isomorphism:

\begin{equation*}
\begin{tikzcd}
    \cdots \rar & Q_2 \dar["\alpha_2"]\rar["d_2"] & Q_1\dar["\alpha_1"] \rar["d_1"] & Q_0\dar["\alpha_0"] \rar & 0\dar \rar & \cdots\\
    \cdots \rar & Q_2\tens_A\underline{A} \dar\rar["d_2\otimes \id"] & Q_1\tens_A\underline{A} \rar["d_1\otimes \id"] \dar["\pr"] & Q_0\tens_A\underline{A} \dar["\pr"]\rar & 0\dar \rar & \cdots\\
    \cdots \rar & 0 \rar[] & Q_1\otimes_A\underline{A}/\Im(d_2\otimes \id)  
    \rar["d_1\otimes \id"] & Q_0\otimes_A\underline{A} \rar & 0 \rar & \cdots.
\end{tikzcd}
\end{equation*}

The projection $\pr :\tensor[_\uB]{Q}{_A} \tens_A \tensor[_A]{\uA}{_A} \to (\tensor[_\uB]{Q}{_A} \tens_A \tensor[_A]{\uA}{_A})_{\subseteq 1}$ of soft truncation is 
an isomorphism on homology in degrees $\leq 1$. 
By \Cref{lem:B_tens_T_homology} we have $H_i(\tensor[_\uB]{Q}{_A}) = 0$ for $i\neq 0,1$.
Therefore it is enough to check that $\alpha$ is an isomorphism on homology in degrees $0$ and $1$.
Moreover, since we are dealing with finite dimensional vector spaces, \Cref{cor:oneside_theorem}\ref{cor:sim:2} gives that for $i=0,1$ the dimensions of $H_i(\tensor[_\uB]{\uB}{_B}\ltensor_B \tensor[_B]{T}{_A})\cong H_i(\tensor[_\uB]{Q}{_A})$ and $H_i(\ \tensor[_\uB]{\!\cT}{_\uA}\!\ltensor_{\underline A} \tensor[_\uA]{\underline{A}}{_A})\cong H_i(\tensor[_\uB] {Q}{_A}\!\otimes_{A}\!\tensor[_A]{\underline{A}}{_A})$ are the same,
making it enough to check that $H_i(\alpha)$ is injective, for $i=0,1$.
\medbreak

Let $i\in \Set{0,1}$, and let $h\in\Ker(d_i)$ be such that $h + \Im(d_{i+1}) \in \Ker(H_i(\alpha))$.
Then there exist morphisms $h_j\in Q_{i+1}$ and $a_j\in {A}$ such that
\begin{align*}
    \Im(d_{i+1}\tens_A \id) \ni \alpha_i(h)
    &= h\tens_A 1  \\
    &= (d_{i+1}\tens_A \id)\Big(\sum_j h_j \tens_A \underline{a_j} \Big)\\
    &= \sum_j (d_{i+1}(h_j)\tens_A \underline{a_j})\\
    &= \sum_j (d_{i+1}(h_ja_j)\tens_A 1)\\
    &=  \Big(\sum_j d_{i+1}(h_ja_j)\Big)\tens_A 1.
\end{align*}

Thus $(h - \sum_j d_{i+1}(h_ja_j)) \otimes_A 1 =0$. 
Letting $\fa$ be the ideal of morphisms in $A$ factoring through a projective object
gives that $h - \sum_j d_{i+1}(h_ja_j) \in Q_i \fa \cap \Ker d_i$.
However, this intersection equals $(\Ker d_i)\fa$, by \Cref{lem:proj_ideal_cap}.
\Cref{lem:proj_ideal_cap} applies for $i=0$ because the cokernel of $Q_0\rightarrow 0$ is $0$ which has the form $\cC(l,y)$ for $y=0$.
\Cref{lem:proj_ideal_cap} applies for $i=1$ because the cokernel of $Q_1\rightarrow Q_0$ is $H_0(Q_A)\cong H_0(\tensor{\underline{B}}{_B}\ltensor_B \tensor[_B]{T}{_A})$
which has the form $\cC(l,m)$ by \Cref{lem:B_tens_T_homology}. 
Since $h - \sum_j d_{i+1}(h_ja_j)\in (\Ker d_i)\fa$ we have

\[
    h - \sum_j d_{i+1}(h_ja_j) = \sum_j \tilde h_j \tilde a_j,
\]
for some $\tilde h_j\in\Ker d_i$ and $\tilde a_j\in\fa$.
By \Cref{lem:B_tens_T_homology} there are isomorphisms over $A^{\op}$:
\[
    \xi_i:H_i(Q_A)\longrightarrow
    \begin{cases}
       \cC(l,m)& i=0,\\
       \cC(l,\Omega m)& i=1,\\
       0 & \text{otherwise.}
    \end{cases}
\]
Notice that $\fa$ annihilates $\Im(\xi_i)$, meaning that  

\[
    \xi_i\Big(\sum_j \tilde h_j \tilde a_j + \Im(d_{i+1}) \Big)
    =\sum_j \xi_i(\tilde h_j +  \Im(d_{i+1})) \tilde a_j
    =0.
\]

Since $\xi_i$ is injective $ h - \sum_j d_{i+1}(h_ja_j)=\sum_j \tilde h_j \tilde a_j\in \Im(d_{i+1})$.
Thus $h\in \Im d_{i+1}$.\qedhere
\end{proof}

\begin{theorem}
    \label{thm:silting_alg}
    The canonical morphisms
    \[ 
        \underline{B} \longrightarrow \End_{\pD(\underline{A}^{\op})} 
        ({}_{\underline{B}}\cT_{\underline{A}})
        \qquad\text{and}\qquad
        \underline{A} \longrightarrow \End_{\pD(\underline{B})} 
        ({}_{\underline{B}}\cT_{\underline{A}}),
    \]
    induced by $-\ltensor_{\underline{B}}{}_\uB\cT_\uA$ and ${}_\uB\cT_\uA\ltensor_{\underline{A}}-$ are isomorphisms.
\end{theorem}
\begin{proof}
    We will show the first isomorphism.
    The second one is done similarly in a symmetric fashion.
    From the isomorphism in \Cref{thm:cTT-iso} the following commutative diagram is produced:

    \begin{equation*}
    \begin{tikzcd}[column sep=4em]
        \End_{\pD(\underline B^{\op})} (\underline{B}_{\underline{B}})\rar["{-\ltensor_{\underline{B}} {}_\uB\underline{B}_B}"]\dar["{-\ltensor_{\underline{B}} {}_\uB\cT_{\underline{A}}}"']
        &\End_{\pD (B^{\op})}(\underline{B}_{B})\dar["{-\ltensor_{B} {}_{B}T_{A}}"]\\
        \End_{\pD(\underline A^{\op})} (\cT_{\underline A})
        \rar["{-\ltensor_{\underline{A}} {}_\uA\underline{A}_A}"]
        &\End_{\pD(A^{\op})}(\cT_{A}).
    \end{tikzcd}
    \end{equation*}

    The morphism induced by $ -\ltensor_{\underline{B}} \tensor[_\uB]{\underline{B}}{_B}$ is an isomorphism, and since ${}_B T_A$ is a tilting complex (by \Cref{lem:t_tilting}), the morphism induced by $-\ltensor_{B} {}_{B}T_{A}$ is also an isomorphism.
    Hence the map induced by $-\ltensor_{\underline{A}} \tensor[_\uA]{\underline{A}}{_A}$ is surjective.
    Thus to show $-\ltensor_{\underline{B}} {}_\uB\cT_{\underline{A}}$ is an isomorphism,
    it is enough to show that
         
    \[
        \End_{\pD \underline A^{\op}} (\cT_{\underline A})\cong
        \End_{\pD A^{\op}}(\cT_{A}),
    \]

    making $-\ltensor_{\underline{A}}\tensor[_\uA]{\underline{A}}{_A}$ an isomorphism.
    This is seen using the following calculation:

    \begin{align*}
        \RHom_{A^{\op}}(\cT_{A},\cT_{A})
        &\cong\RHom_{A^{\op}}\big(\cT_{\underline{A}} \ltensor_{\underline A}\tensor[_{\underline A}]{\underline{A}}{_{A}},\RHom_{\underline{A}^{\op}}(\tensor[_A]{\underline{A}}{_{\underline A}} ,\tensor{\cT}{_{\underline{A}}}\!)\big)\\
        &\cong\RHom_{\underline{A}^{\op}}\big((\tensor{\cT}{_{\underline{A}}}\! \ltensor_{\underline A}{}_{\underline A}\tensor{\underline{A}}{_{A}})\ltensor_A \tensor[_{A}]{\underline{A}}{_{\underline A}} ,\tensor{\cT}{_{\underline{A}}}\!\big)\\
        &\cong\RHom_{\underline A^{\op}}\big(\tensor{\underline{B}}{_{B}}\ltensor_{B} \tensor[_B]{T}{_A}\ltensor_A \tensor[_{ A}]{\underline{A}}{_{\underline A}} ,\tensor{\cT}{_{\underline{A}}}\!\big)&\text{(by \Cref{cor:oneside_theorem}\ref{cor:sim:2})}\\
        &\cong\RHom_{\underline A^{\op}}\big(\tensor{\cP}{_A}\tens_A\tensor[_{A}]{\underline{A}}{_{\underline A}} ,\tensor{\cT}{_{\underline{A}}}\!\big)&\hspace{-40pt}\text{(by \Cref{cor:b_tens_T_1} and \Cref{lem:projres_cP})}\\
        &\cong\RHom_{\underline A^{\op}}\big(P_\uA\oplus P_\uA[3] ,\cT_{\underline{A}}\big).&\text{(by \Cref{lem:wif}\ref{lem:item:wif3})}
    \end{align*}

    By taking the 0'th homology one obtains that
    \begin{align*}
        \End_{\pD(A^{\op})}(\cT_{A}) 
        &\cong \Hom_{\pD(\underline A^{\op})}\big(P_\uA\oplus P_\uA[3] ,\cT_{\underline{A}}\big)
        \cong \Hom_{\pD(\underline A^{\op})}\big(P_\uA ,\cT_{\underline{A}}\big)
        \cong \End_{\pD(\underline A^{\op})}\big(\cT_{\underline{A}} \big),
    \end{align*}

    where the second isomorphism comes from $\cT_{\underline{A}}$ being isomorphic in $\pD(\underline{A}^{\op})$ to a two term complex of projective objects, see  \Cref{cor:oneside_theorem}\ref{cor:sim:1}, and the last isomorphism is also by \Cref{cor:oneside_theorem}\ref{cor:sim:1}.\qedhere
\end{proof}

\begin{lemma}
    \label{lem:P_tilting}
The complex $P_\uA$,
\[
    \begin{tikzcd}[column sep=2em]
        \cC(l,l_1) \rar["{\phi_1}_*"] &
        \cC(l,l_0),
    \end{tikzcd}
\]
is a two-term tilting complex in $\pD(\underline{A}^{\op})$. 
\end{lemma}
\begin{proof}
    $P_\uA$ is a silting complex by \cite[thm 2.18 and rmk 2.19(2)]{august2020finiteness}. 
    Thus to show that $P_\uA$ is a tilting complex
    it is enough to show that $\Hom_{\pK^{b}(\proj \underline{A}^{\op})}(P_\uA, P_\uA[-1])=0$.
    \medbreak

    Recall that $P_\uA$ is defined using the triangle

    \begin{equation*}
    \begin{tikzcd}
        l_1\rar["\phi_1"]&
        l_0\rar["\phi_0"]&
        m\rar["\Sigma\phi_2"]&
        \Sigma l_1.
    \end{tikzcd}
    \end{equation*}

    By applying $\cC(l, -)$ to this triangle we obtain an exact sequence
    \begin{equation*}
    \begin{tikzcd}
        0\rar&
    \cC(l,\Sigma^{-1}m)\rar["{\phi_2}_*"]&
        \cC(l,l_1)\rar["{\phi_1}_*"]&
        \cC(l,l_0)\rar["{\phi_0}_*"]&
        \cC(l,m)\rar&
        0.
    \end{tikzcd}
    \end{equation*}

    This shows that $\Ker({\phi_1}_*) = \cC(l,\Sigma^{-1}m)$, and 
    $\Coker({\phi_1}_*) = \cC(l,m)$. 
    Let a chain map $f\in\Hom_{\pK^{b}(\proj \underline{A}^{\op})}(P_\uA, P_\uA[-1])$ be given:
    \begin{equation*}
    \begin{tikzcd}
        \cC(l,l_1)\dar\ar[r,"{\phi_1}_*"] & \cC(l,l_0)\ar[d, "f_0"]\rar&0\dar\\
        0\rar&\cC(l,l_1) \rar["{\phi_1}_*"] & \cC(l,l_0).
    \end{tikzcd}
    \end{equation*}

    $f$ being a chain map implies that $f_0{\phi_1}_{*}=0$. Thus $f_0$ factors through $\Coker({\phi_1}_*)$, say $f_0 = f_0' {\phi_0}_{*}$.
    If we use that $f$ is a chain map again we get that ${\phi_1}_{*}f_0'$=0, giving that $f_0'$ factors through $\Ker({\phi_1}_*)$, say $f_0' = {\phi_2}_{*}f_0''$, with $f_0''\in \Hom_{\underline{A}}(\cC(l,m), \cC(l,\Sigma^{-1} m))$.
    Therefore $f_0 = {\phi_2}_{*}f_0'' {\phi_0}_*$:
    \begin{equation*}
    \begin{tikzcd}
        \cC(l,l_1)\ar[r,"{\phi_1}_*"] & \cC(l,l_0)\ar[d, "f_0"{name=Z}, crossing over]\rar[twoheadrightarrow, "{\phi_0}_*"]&\cC(l,m)
     \arrow[dll, "f_0''", rounded corners,
    to path={ -- ([xshift=2ex]\tikztostart.east)
    |- ([yshift=-1ex]Z) [yshift=0.2ex]\tikztonodes
    -| ([xshift=-2ex]\tikztotarget.west) -- (\tikztotarget)}] \\
        \cC(l,\Sigma^{-1}m)\rar[rightarrowtail,"{\phi_2}_*"]&\cC(l,l_1) \rar["{\phi_1}_*"] & \cC(l,l_0).
    \end{tikzcd}
    \end{equation*}

    Since $m$ is rigid and $\Sigma^2 m \cong m$ we have that $\cC(m, \Sigma^{-1}m)=0$.
    Thus \Cref{lem:wif}\ref{lem:item:wif6} gives that $f_0''=0$, and thereby $f_0=0$.
    Hence $f = 0$, making $P_\uA$ a tilting complex.
\end{proof}

\begin{corollary}
    \label{cor:Ptilting}
    ${}_\uB\cT_\uA$ is a two-sided tilting complex. In particular 
   \[
       -\ltensor_{\underline{B}} {}_\uB\cT_\uA: \pD(\underline{B}^{\op})\longrightarrow \pD(\underline{A}^{\op})
   \]
    is a triangulated equivalence, making $\underline{A}$ and $\underline{B}$ derived equivalent.
\end{corollary}
\begin{proof}
    Recall that $\cT_\uA\cong P_\uA$ in $\pD(\underline{A}^{\op})$ (by \Cref{cor:oneside_theorem}\ref{cor:sim:1}), and that $P_\uA$ is a tilting complex (by \Cref{lem:P_tilting}). Symmetrically ${}_{\underline{B}}\cT$ is also a tilting complex.
    This establishes the second half of part (1) as well as parts (2) and (3) of \Cref{def:tilting-complex}.
    The first half of part (1) holds by \Cref{thm:silting_alg}.
    The triangulated equivalence follows from \cite[prop. 8.1.4]{Keller1998}.
\end{proof}

\section{Examples}
\subsection{Cluster-tilting objects from \texorpdfstring{$\pazocal{C}(D_{2n})$}{C(D2n)}}\leavevmode
\medbreak

    It was shown in \cite[lem. 4.5]{bastian2014towards} that the self-injective cluster tilted algebras of the cluster category $\cC=\pazocal{C}(D_{2n})$ are derived equivalent.
    This was done by finding a tilting complex ad hoc. 
    In this example we will see how our results can be used to find such a tilting complex, and thereby recover the tilting complex from \cite{bastian2014towards}.
    \medbreak

    To apply the results from the previous section on this example,
    we need to ensure that $\cC$ has a Frobenius model $\cE$ that satisfies \Cref{set:EC}

    \begin{theorem}
        There exists a Frobenius category $\cE$, such that $\cC = \underline{\cE}$.
        Furthermore the pair of $\cE$ and $\cC$ satisfy \Cref{set:EC}.
    \end{theorem}
    \begin{proof}
        Let $I$ be the direct sum of all injective indecomposable objects in $\mod(kD_{2n})$,
        and $M = I \oplus \tau I$.
        This $M$ has the needed properties to apply \cite[thm. 2.1]{geiss2007cluster},
        giving a Frobenius category $\cE$ such that $\cC= \underline{\cE}$, which also satisfies \Cref{set:EC}.
    \end{proof}

    Let $n\in\NN$, with $n\geq 4$. Consider the quiver $D_{2n}$.

    \begin{equation*}
    \begin{tikzcd}
        &&&&a\ar[dl]\\
        c\rar&
        \bullet\rar&
        \cdots\rar&
        \bullet&\\
        &&&&b\ar[ul]
    \end{tikzcd}
    \end{equation*}
    
    There are indecomposable projective representations $P(a), P(b), P(c)$ corresponding to the vertices $a,b,c$,
    and these can be viewed as objects of $\cC$.
    In \cite{ringel2008self}, Ringel described two cluster-tilting objects in $\cC$, see \Cref{fig:AR_d8_2}:

    \begin{equation*}
        T_1 = 
        \left(\bigoplus^{n-1}_{i=0} \tau^{-2i}P(a) \right) \oplus
        \left(\bigoplus^{n-1}_{i=0} \tau^{-2i-1}P(b) \right) 
    \end{equation*}
    \vspace{.5em}
    \begin{equation*}
        T_2 = 
        \left(\bigoplus^{n-1}_{i=0} \tau^{-2i}P(a) \right) \oplus
        \left(\bigoplus^{n-1}_{i=0} \tau^{-2i-1}P(c) \right).
    \end{equation*}
    \vspace{0.1em}

    Denote their endomorphism algebras by $A_i = \End_\cC(T_i) = \cC(T_i,T_i)$.
    To describe the endomorphism algebras we define the following quivers:

\begin{minipage}{0.45\textwidth}
\begin{center}
\begin{tikzpicture}
    \def\n {4}
    \def\rad{1.3}
    \def\soffset{15}
    \def\toffset{15}
    \node at (180:2.6) {$Q_1:$};
    \foreach \i in {1,...,\n}{
        \ifthenelse{\i=1}{\node[yshift=-.1em]}{\node} at ({(\i*360)/\n}:\rad) {\ifthenelse{\i=1}{$\cdots$}{
                \ifthenelse{\i=2}{$2$}{ 
                    \ifthenelse{\i=3}{$1$}{
                        \ifthenelse{\i=4}{$2n$}{
                        }
                    }
                }
            }};
        \node at ({(\i*360)/\n+(360)/(2*\n)}:{\rad+.3}) {$\alpha$};
        \def\soffsett{\toffset}
        \def\toffsett{\soffset}
        \ifthenelse{\i=1}{\def\soffset{22}}{}
        \ifthenelse{\i=\n}{\def\toffset{19}}{}
        \draw[->] ({(\i*360)/\n + \soffset}:\rad) 
            arc ({(\i*360)/\n+\soffset}:{((\i+1)*360)/\n-\toffset}:\rad);
        \def\toffset{\toffsett}
        \def\soffset{\soffsett}
    }
\end{tikzpicture}
\end{center}
\end{minipage}
\begin{minipage}{0.45\textwidth}
\begin{center}
\begin{tikzpicture}
   \def\n{5} 

   \def\rot{-360/10} 
   \def\innerRad{1.3} 
   \def\outerRad{2.3} 
   \def\bend{15} 
   \def\offset{.8em}
   \def\missing{3}

   \node at (180:3.4) {$Q_2:$};
   \foreach \i in {1,...,\n}{
        \node at ({(\i-2)*360/\n+360/(4*\n)+\rot}:\innerRad) {\ifthenelse{\i=\n}{$n$}{\ifthenelse{\i=4}{$n-1$}{$\i$}}};
        \ifthenelse{\i=\missing}{}{
            \node at ({(\i-2)*360/\n+(360/(\n*2))+360/(4*\n) +\rot}:\outerRad) {$\bullet$};
            \draw[shorten >=\offset,shorten <=\offset, ->] ({(\i-2)*360/\n+360/(4*\n) + \rot}:\innerRad) arc ({(\i-2)*360/\n+360/(4*\n)+\rot}:{(\i-1)*360/\n+360/(4*\n)+\rot}:\innerRad);
            \node[text width=1em] at ({(\i-2)*360/\n+360/(4*\n)+360/(2*\n)+\rot}:1.1) {$\alpha$};
            \node[text width=1em] at ({(\i-2)*360/\n+2*360/(4*\n)+\rot}:{(\innerRad+\outerRad)/2+.3}) {$\beta$};
            \node[text width=1em] at ({(\i-2)*360/\n+4*360/(4*\n)+\rot+3}:{(\innerRad+\outerRad)/2+.2}) {$\beta$};
            \ifthenelse{\i=4}{
            \draw[shorten <=\offset, ->] ({(\i-2)*360/\n+360/(\n*2)+360/(4*\n)+\rot}:\outerRad) to [bend left=\bend] ({(\i-2)*360/\n+360/(4*\n)+\rot+15}:{\innerRad+0.38});
            }{
            \draw[shorten >=\offset,shorten <=\offset, ->] ({(\i-2)*360/\n+360/(\n*2)+360/(4*\n)+\rot}:\outerRad) to [bend left=\bend] ({(\i-2)*360/\n+360/(4*\n)+\rot}:\innerRad);
            }
            \draw[shorten >=\offset,shorten <=\offset, <-] ({(\i-2)*360/\n+360/(\n*2)+360/(4*\n)+\rot}:\outerRad) to [bend right=\bend] ({(\i-1)*360/\n+360/(4*\n)+\rot}:\innerRad);
        }
   }
   \node at ({(\missing-2)*360/\n+(360/(\n*2))+360/(4*\n) +\rot}:{(\innerRad)}) {$\cdots$};
\end{tikzpicture}
\end{center}
\end{minipage}
 
Then $A_i = kQ_i/I_i$, with $I_1 = \langle\alpha^{2n-1}\rangle$, and $I_2 = \langle \alpha\beta,\beta\alpha,\beta^2-\alpha^{n-1} \rangle$.
Both of these algebras are self-injective by \Cref{lem:hom_sym}, but neither of them is symmetric. 
Now \Cref{cor:Ptilting} says that $A_1$ and $A_2$ are derived equivalent.
Furthermore, there is a direct way to calculate the associated one-sided tilting complexes.

\bigbreak

\begin{figure}[ht]
\begin{equation*}
\begin{tikzcd}[column sep=.8em,row sep=.8em, nodes in empty cells]
    &\bullet\ar[dr]& &\clap{{\tiny $\tau^{-1}P(a)$}}\ar[dr]& &\bullet\ar[dr]& &\bullet\ar[dr]& &\bullet\ar[dr]& &\bullet\ar[dr]& &\bullet& &&\\
    \bullet\ar[dr]\ar[ur]\ar[r]&\smash{\hspace{0.5em}\clap{{\tiny $P(b)$}}}\phantom{\bullet}\ar[r]& \bullet\ar[dr]\ar[ur]\ar[r]&\bullet\ar[r]& \bullet\ar[dr]\ar[ur]\ar[r]&\bullet\ar[r]& \bullet\ar[dr]\ar[ur]\ar[r]&\bullet\ar[r]& \bullet\ar[dr]\ar[ur]\ar[r]&\bullet\ar[r]& \bullet\ar[dr]\ar[ur]\ar[r]&\bullet\ar[r]& \bullet\ar[dr]\ar[ur]\ar[r]&\bullet&&&\\
     &\bullet\ar[ur]\ar[dr]& &\bullet\ar[ur]\ar[dr]& &\bullet\ar[ur]\ar[dr]& &\bullet\ar[ur]\ar[dr]& &\bullet\ar[ur]\ar[dr]& &\bullet\ar[ur]\ar[dr]& &\bullet\ar[dr]& &&\\
    && \bullet\ar[ur]\ar[dr]&& \bullet\ar[ur]\ar[dr]&& \bullet\ar[ur]\ar[dr]&& \bullet\ar[ur]\ar[dr]&& \bullet\ar[ur]\ar[dr]&& \bullet\ar[ur]\ar[dr]&& \bullet\ar[dr]&&\\\
    && &\bullet\ar[ur]\ar[dr]& &\bullet\ar[ur]\ar[dr]& &\bullet\ar[ur]\ar[dr]& &\bullet\ar[ur]\ar[dr]& &\bullet\ar[ur]\ar[dr]& &\bullet\ar[ur]\ar[dr]& &\bullet\ar[dr]&\\
    &&&& \bullet\ar[ur]\ar[dr]&& \bullet\ar[ur]\ar[dr]&& \bullet\ar[ur]\ar[dr]&& \bullet\ar[ur]\ar[dr]&& \bullet\ar[ur]\ar[dr]&& \bullet\ar[ur]\ar[dr]&&\bullet\ar[dr]\\
    &&&& &\bullet\ar[ur]& &\bullet\ar[ur]& &\clap{{\tiny $\tau^{-2}P(c)$}}\ar[ur]& &\bullet\ar[ur]& &\bullet\ar[ur]& &\bullet\ar[ur]&&\bullet
\end{tikzcd}
\end{equation*}
    \caption{Auslander-Reiten quiver of $\mod(kD_{8})$.}%
    \label{fig:AR_d8_2}
\end{figure}
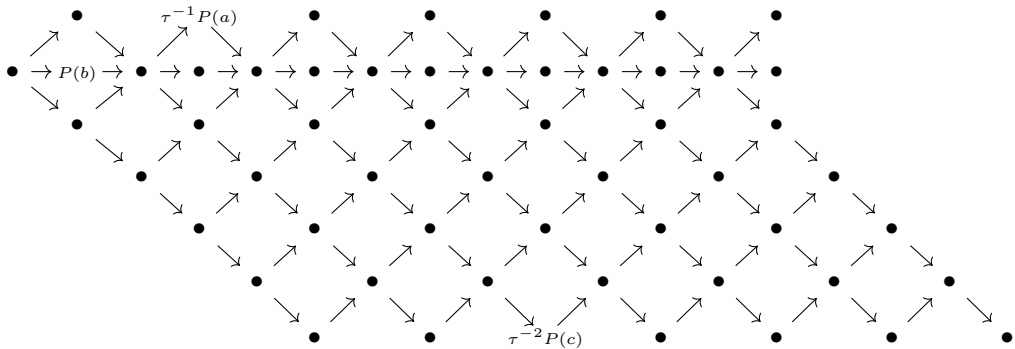

Using the Auslander-Reiten quiver of $\mod(kD_{2n})$ (see \Cref{fig:AR_d8_2} for an example in the case $n=4$),
one can use the dimension vectors to see that there is an exact sequence (see \cite[prop.~IX.3.1, lem.~IX.1.1(a)]{assem2006elements})

\begin{equation*}
\begin{tikzcd}
    0\rar&
    P(b)\rar["\phi"]&
    \tau^{-1}P(a)\rar["\psi"]&
    \tau^{-2}P(c)\rar&
    0
\end{tikzcd}
\end{equation*}

inducing a triangle
\begin{equation*}
\begin{tikzcd}
    P(b)\rar["\phi"]&
    \tau^{-1}P(a)\rar["\psi"]&
    \tau^{-2}P(c)\rar&
    \Sigma P(b)
\end{tikzcd}
\end{equation*}

in the cluster category $\cC$.This implies that for each $i$ there is a triangle

\begin{equation*}
\begin{tikzcd}
    \tau^{i}P(b)\rar["\phi_i"]&
    \tau^{i-1}P(a)\rar["\psi_{i-1}"]&
    \tau^{i-2}P(c)\rar&
    \Sigma\tau^{i} P(b)
\end{tikzcd}
\end{equation*}

in $\cC$.
With this we construct the following triangle.
Let $\Phi' = \oplus_{i=0}^{n-1} \phi_{-2i-1}$ and let $\Psi' = \oplus_{i=1}^{n} \psi_{-2i}$,
then there exists a triangle

\begin{equation*}
    \begin{tikzcd}[ampersand replacement=\&, column sep=3.8em]
        \left(\bigoplus^{n-1}_{i=0} \tau^{-2i-1}P(b) \right)
        \rar["{\begin{pmatrix} 0\\\Phi'\end{pmatrix} }"]\&
    \begin{matrix}
     \left(\bigoplus^{n}_{i=1} \tau^{-2i}P(a) \right) \\\oplus\\
     \left(\bigoplus^{n}_{i=1} \tau^{-2i}P(a) \right) 
    \end{matrix}
        \rar["{\begin{pmatrix} \id&0\\0&\Psi' \end{pmatrix} }"]\&
    \begin{matrix}
     \left(\bigoplus^{n}_{i=1} \tau^{-2i}P(a) \right)\\\oplus\\
     \left(\bigoplus^{n}_{i=1} \tau^{-2i-1}P(c) \right)
 \end{matrix}\rar\&\hbox{}\\
        T_1^2\rar["\Phi"]\uar[equal]\&
        T_1^1\rar["\Psi"]\uar[equal]\&
        T_2\uar[equal]\rar\&\hbox{}
\end{tikzcd}
\end{equation*}

with $T_1^j\in \add(T_1)$. Notice that we have used that $\tau^{2n}\cong \id$ when describing $T_2$.
It follows from $T_1$ being rigid that $\Psi$ is an $\add(T_1)$ pre-cover, namely if there is a morphism $\tilde\Psi:S\rightarrow T_2$, with $S\in \add(T_1)$ then $\cC(S,\Sigma T_1^2) = 0$, and $\tilde\Psi$ will therefore factor through $\Psi$. 
To see that $\Psi$ is a cover, it is enough to check that $\Phi\in\rad_\cC$ (see \cite[lem 3.12]{fedele2019}).
All the components $\phi_{-2i-1}$ of $\Phi$ are morphisms between two different indecomposable objects, and are therefore all in the radical.
Thus $\Phi$ is in the radical.

\bigbreak
Define the following complex concentrated in degrees 0,1:
\begin{equation*}
    P_1: 
\begin{tikzcd}
    \cC(T_1,T_1^2) \rar["\Phi_*"]&
    \cC(T_1,T_1^1).
\end{tikzcd}
\end{equation*}

This complex is a tilting complex over $A_1$ by \Cref{lem:P_tilting}.
By \Cref{thm:silting_alg}, there is an isomorphism $A_2\cong \End_{\pD(A_1^{\op})}(P_1)$.
Similarly a tilting complex $P_2$ could be found such that $A_1\cong \End_{\pD(A_2^{\op})}(P_2)$.

\bigbreak
To verify that this is indeed the case we check the isomorphism $A_2\cong \End_{\pD(A_1^{\op})}(P_1)$. 
For each $T'\in\add(T_1)$, the right $A_1$-module $\cC(T_1, T')$ is projective. We can therefore identify the Hom-space from $T_1$ to indecomposable summands of $T_1$ with indecomposable projective modules over $A_1$.
We will do that as follows:
\begin{equation*}
    P_{A_1}(2i) \cong \cC(T_1, \tau^{-2i}P(a))
    \qquad\text{and}\qquad 
    P_{A_1}(2j+1) \cong \cC(T_1, \tau^{-(2j+1)}P(b))
\end{equation*}

for $0< i\leq n$, and $0\leq j < n$. The morphisms between projective objects can be described as follows:

\begin{equation*}
    \dim \Hom(P_{A_1}(i), P_{A_1}(j)) = \begin{cases}
        0&\text{ if } j = i - 1\\
        0&\text{ if } i = 1,\ j=2n\\
        1&\text{ otherwise.}\\
    \end{cases}
\end{equation*}

Now $\End_{\pD(A_1^{\op})}(P_1)$ can be calculated. There are $2n$ indecomposable components in $P_1$:
\begin{alignat*}{2}
    B_i:\quad& P_{A_1}(2i+1) \,&&\xrightarrow{\phi_{-2i-1*}}\,P_{A_1}(2i+2),\\ 
    C_j:\quad& \phantom{\tau^{-2i-1}P(,}0\,&&\xrightarrow{\phantom{\phi_{-2i-1*}}}\,P_{A_1}(2j),
\end{alignat*}

with $0\leq i < n$, and $0 < j \leq n$. To describe the morphisms we split it into cases.

\begin{enumerate}[label=({\roman*})]
    \item $\gamma\in\Hom(C_i,B_{i-1})\neq 0$, for $0<i \leq n$. 
        Here $\gamma$ is the inclusion.
    \item $\beta\in\Hom(B_i,C_{i})\neq 0$, for $0 < i \leq n-1$ and $\beta\in\Hom(B_0,C_{n})\neq 0$.
        Here $\beta$ is induced by the morphism $P(a)\rightarrow \tau^{2}P(a)$.
    \item $\alpha\in\Hom(C_i,C_{i+1})\neq 0$, for all $0< i \leq n-1$, and $\alpha\in\Hom(C_{n},C_{1})\neq 0$.
        Here $\alpha$ is induced by the morphism $P(a)\rightarrow \tau^{-2}P(a)$.
    \item $\Hom(B_i,B_{i+1}) = 0$, for all $0\leq i < n-1$, and $\Hom(B_{n-1},B_{0}) = 0$.
        This is due to the ‘potential’ morphisms being null-homotopic.
    \item \label{ex:rel:alphabeta}
        $\delta\in\Hom(C_i,C_{i-1})\neq 0$, for all $1< i \leq n$, and $\delta\in\Hom(C_{1},C_{n})\neq 0$.
        But notice that $\delta$ factors through $B_{i-1}$:

        \begin{equation*}
        \begin{tikzcd}
            C_i\dar["\gamma"]&&0 \rar\dar&P_{A_1}(2i)\dar\\
            B_{i-1}\dar["\beta"]&:&P_{A_1}(2i-1) \rar\dar& P_{A_1}(2i)\dar\\
            C_{i-1}&&0 \rar&P_{A_1}(2i-2).
        \end{tikzcd}
        \end{equation*}
\end{enumerate}

From this we can determine the quiver of the endomorphism algebra $\End_{\pD(A_1^{\op})}(P_1)$:
\begin{center}
\begin{tikzpicture}
   \def\n{4} 

   \def\rot{360/15} 
   \def\innerRad{1.3} 
   \def\outerRad{2.3} 
   \def\bend{15} 
   \def\offset{.8em}
   \def\missing{2}

   \node at (180:3.4) {$Q:$};
   \foreach \i in {0,...,\n}{
        \node at ({(\i-2)*360/\n+360/(4*\n)+\rot}:\innerRad) {\ifthenelse{\i=3}{$C_{n-1}$}{\ifthenelse{\i=4}{}{$C_{\i}$}}};
        \ifthenelse{\i=\missing}{}{

            \node at ({(\i-2)*360/\n+(360/(\n*2))+360/(4*\n) +\rot}:\outerRad) {\ifthenelse{\i=3}{$B_{n-1}$}{\ifthenelse{\i=0}{$B_{0}$}{\ifthenelse{\i=1}{$B_1$}{}}}};

            \draw[shorten >=\offset,shorten <=\offset, ->] ({(\i-2)*360/\n+360/(4*\n) + \rot}:\innerRad) arc ({(\i-2)*360/\n+360/(4*\n)+\rot}:{(\i-1)*360/\n+360/(4*\n)+\rot}:\innerRad);
            \node[text width=1em] at ({(\i-2)*360/\n+360/(4*\n)+360/(2*\n)+\rot}:1.1) {$\alpha$};
            \node[text width=1em] at ({(\i-2)*360/\n+2*360/(4*\n)+\rot}:{(\innerRad+\outerRad)/2+.3}) {$\beta$};
            \node[text width=1em] at ({(\i-2)*360/\n+4*360/(4*\n)+\rot+3}:{(\innerRad+\outerRad)/2+.2}) {$\gamma$};
            \ifthenelse{\i=3}{
                \draw[shorten >=1.5em,shorten <=\offset, ->] ({(\i-2)*360/\n+360/(\n*2)+360/(4*\n)+\rot}:\outerRad) to [bend left=\bend] ({(\i-2)*360/\n+360/(4*\n)+\rot}:\innerRad);
            }{
                \draw[shorten >=\offset,shorten <=\offset, ->] ({(\i-2)*360/\n+360/(\n*2)+360/(4*\n)+\rot}:\outerRad) to [bend left=\bend] ({(\i-2)*360/\n+360/(4*\n)+\rot}:\innerRad);
            }
            \draw[shorten >=\offset,shorten <=\offset, <-] ({(\i-2)*360/\n+360/(\n*2)+360/(4*\n)+\rot}:\outerRad) to [bend right=\bend] ({(\i-1)*360/\n+360/(4*\n)+\rot}:\innerRad);
        }
   }
   \node at ({(\missing-2)*360/\n+(360/(\n*2))+360/(4*\n) +\rot}:{(\innerRad)}) {$\cdots$};
\end{tikzpicture}
\end{center}

It is straightforward to check that $\beta\gamma\cong \alpha^{n-1}$ using \ref{ex:rel:alphabeta}. 
It follows from the form of $I_1$ that $\alpha\beta = 0$.  Lastly there is the relation $\gamma\alpha = 0$, which is due to $\gamma\alpha$ being null-homotopic. 
It follows from $\alpha^{n-1} \neq 0$ and $\beta\gamma \neq 0$, that there are no more relations. This now means that

\begin{equation*}
    \End_{\pD(A_1^{\op})}(P_1) \cong Q / \langle \alpha^{n-1}-\beta\gamma, \alpha\beta,\gamma\alpha \rangle \cong A_2.
\end{equation*}

\subsection{Example from symmetric \texorpdfstring{$(k,n)$}{(k,n)}-Postnikov diagrams}\leavevmode
\medbreak

It turns out that a good source of examples for finding derived equivalent algebras using \Cref{cor:Ptilting}
are Postnikov diagrams as shown in \Cref{fig:post39_1,fig:post39_2}.
A $(k,n)$-Postnikov diagram $D$ is a Postnikov diagram with $n$ vertices, and strands going from vertices $i$ to vertices $i+k$.
Such a diagram is called \emph{symmetric} if it is invariant under rotation by $k$ vertices, see
\Cref{fig:post39_1,fig:post39_2}.
Furthermore, $D$ is called \emph{reduced} if no ``untwisting'' moves can be applied. 
For a detailed description of these properties see \cite[sec. 4]{pasquali2020self}.
\medbreak

To each Postnikov diagram $D$ one can associate an ice quiver with potential
$(Q,W,F)$, where 
\begin{itemize}
    \item $Q$ is the quiver associated to $D$. The vertices of $Q$ corresponds to the alternating regions of $D$.
        There is an arrow between two vertices if their corresponding regions meet at an intersection of strands,
        the arrow will point with the `flow' of those intersecting stands.
        See \Cref{fig:post39_1,fig:post39_2}.
    \item $W$ is the potential given by the sum of clockwise cycles in $Q$ minus anti-clockwise cycles in $Q$, and
    \item $F$ is the frozen vertices, which is the set of vertices on the boundary of $D$.
\end{itemize}
For further details see \cite[sec. 4]{pasquali2020self}.
Denote the associated frozen Jacobian algebra $\pazocal{P}(Q,W,F)$.
\medbreak

Next we construct the boundary algebra as described in \cite[sec. 6]{pasquali2020self}. Given $k,n\in \NN$, with $k<n$, consider a $\ZZ/n\ZZ$-grading on $\CC[x,y]$ given by
$\deg x = 1$, and $\deg y = -1$.
Now let $R = \CC[x,y]/(x^{k} - y^{n-k})$ and define the boundary algebra

\begin{equation*}
    B = \End_R^{\ZZ/n\ZZ}\left( \bigoplus_{i\in \ZZ/n\ZZ} R(i) \right),
\end{equation*}

where $(i)$ indicates the shift in degrees by $i$ in $\mod^{\ZZ/n\ZZ}(R)$.
Let $\hat B$ be the completion of $B$ with respect to the ideal $(x,y)$.

\indent The following two theorems are a collection of results due to Geiß-Jensen-King-Leclerc-Pasquali-Schr{\"o}er-Su.
The first theorem shows \Cref{set:EC} is satisfied.
The category $\CM$ of Cohen-Macaulay modules and its stabilisation $\uCM$ were introduced in \cite[secs. 3, 4]{jensen2016cat}.

\begin{theorem} Let $k,n\in\NN$ with $k<n$. We then have the following results
    \begin{enumerate}
        \item $\CM(\hat B)$ is a Frobenius category (\cite[cor.3.7]{jensen2016cat})
        \item $\uCM(\hat B)$ is 2-Calabi--Yau (\cite[cor. 4.6]{jensen2016cat} and \cite[prop. 3.4]{geiss2008partial})
        \item $\uCM(\hat B)$ is Hom-finite (\cite[cor. 4.6]{jensen2016cat} and \cite[sec. 3.1, 3.2]{geiss2008partial}).
        \item $\uCM(\hat B)$ has split idempotents.
    \end{enumerate}
\end{theorem}
\begin{proof}
    That $\uCM(\hat B)$ has split idempotents comes from the fact that $\Sub Q_k$ from \cite[cor. 4.6]{jensen2016cat} and \cite[sec. 3]{geiss2008partial} has split idempotents. This is due to it being the full subcategory of submodules of sums of $Q_k$, in a module category which has split idempotents.
    Since $\Sub Q_k$ has split idempotents, $\uSub Q_k$ has split idempotents, and by \cite[cor. 4.6]{jensen2016cat} there is a triangle equivalence $\uSub Q_k\cong\uCM(\hat B)$.
\end{proof}
The next collection of result describes the objects we want to work with.
\begin{theorem} We have the following results
    \label{thm:pasquali}
    \begin{enumerate}
        \item For every reduced $(k,n)$-Postnikov diagram $D$, there is an associated cluster-tilting object $T(D)$ of $\CM(\hat B)$ (\cite[thm. 7.2]{pasquali2020self}).
        \item \label{p:sinj} Given a reduced $(k,n)$-Postnikov diagram $D$, then $D$ is symmetric if and only if the endomorphism ring $\End(T(D))$ is self-injective (\cite[thm. 8.2, lem 7.7]{pasquali2020self}). 
        \item \label{p:jac} If $D$ is a $(k,n)$-Postnikov diagram, $T$ the associated cluster tilting object, (Q,W,F) the associated ice quiver with potential, then $\underline{\End}_{\hat B}(T)\cong \pazocal{P}(Q,W,F)/\langle F\rangle$, where $\langle F\rangle$ is the ideal generated by the frozen vertices (\cite[lem. 7.5, prop 7.6, sec. 3]{pasquali2020self}).
    \end{enumerate}
\end{theorem}

\begin{corollary}
    \label{cor:postnikov_example}
    Let $k,n\in \NN$, with $k<n$. 
    Let $D,D'$ be two symmetric and reduced $(k,n)$-Postnikov diagrams, with
    associated cluster tilting objects $T = T(D)$ and $T' = T(D')$ in $\uCM(\hat B)$. 
    Then $\uEnd(T)$ and $\uEnd(T')$ are derived equivalent.
\end{corollary}
\begin{proof}
    Since $D$ and $D'$ are symmetric Postnikov diagrams \Cref{thm:pasquali}(\ref{p:sinj}) gives that $\uEnd{T}$ and $\uEnd{T'}$ are self-injective. Therefore \Cref{cor:Ptilting} gives that $\uEnd(T)$ and $\uEnd(T')$ are derived equivalent.
\end{proof}

As an example of the use of \Cref{cor:postnikov_example}, see \Cref{fig:post39_1,fig:post39_2} for two symmetric and reduced $(3,9)$-Postnikov diagrams and their associated ice quivers. 
The frozen vertices are exactly the ones on the boundary.
Denote these ice quivers with potential by $(Q,W,F)$ and $(Q',W',F')$ respectively.
Using Theorems \ref{thm:pasquali}(\ref{p:sinj}) and \ref{thm:pasquali}(\ref{p:jac}) together with \Cref{cor:postnikov_example} we get that
the associated algebras $\pazocal{P}(Q,W,F)/\langle F\rangle$ and $\pazocal{P}(Q',W',F')/\langle F'\rangle$ are self-injective and derived equivalent.
Note that in contrast to \cite{august2020tilting} the 2-CY algebras considered here are not symmetric.

\begin{figure}[H]
    \includegraphics[scale=0.6, draft=false]{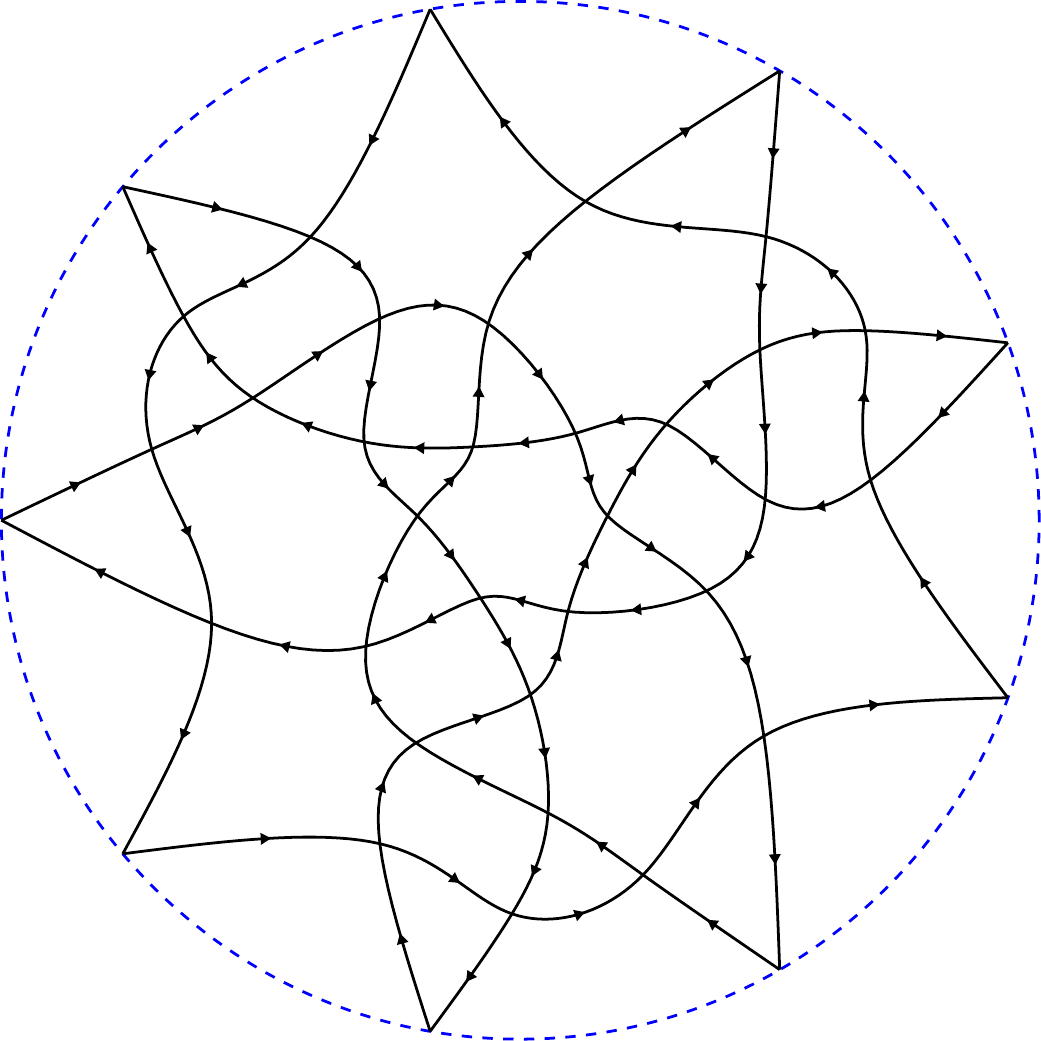}
    \hspace{10pt}
    \includegraphics[scale=0.6, draft=false]{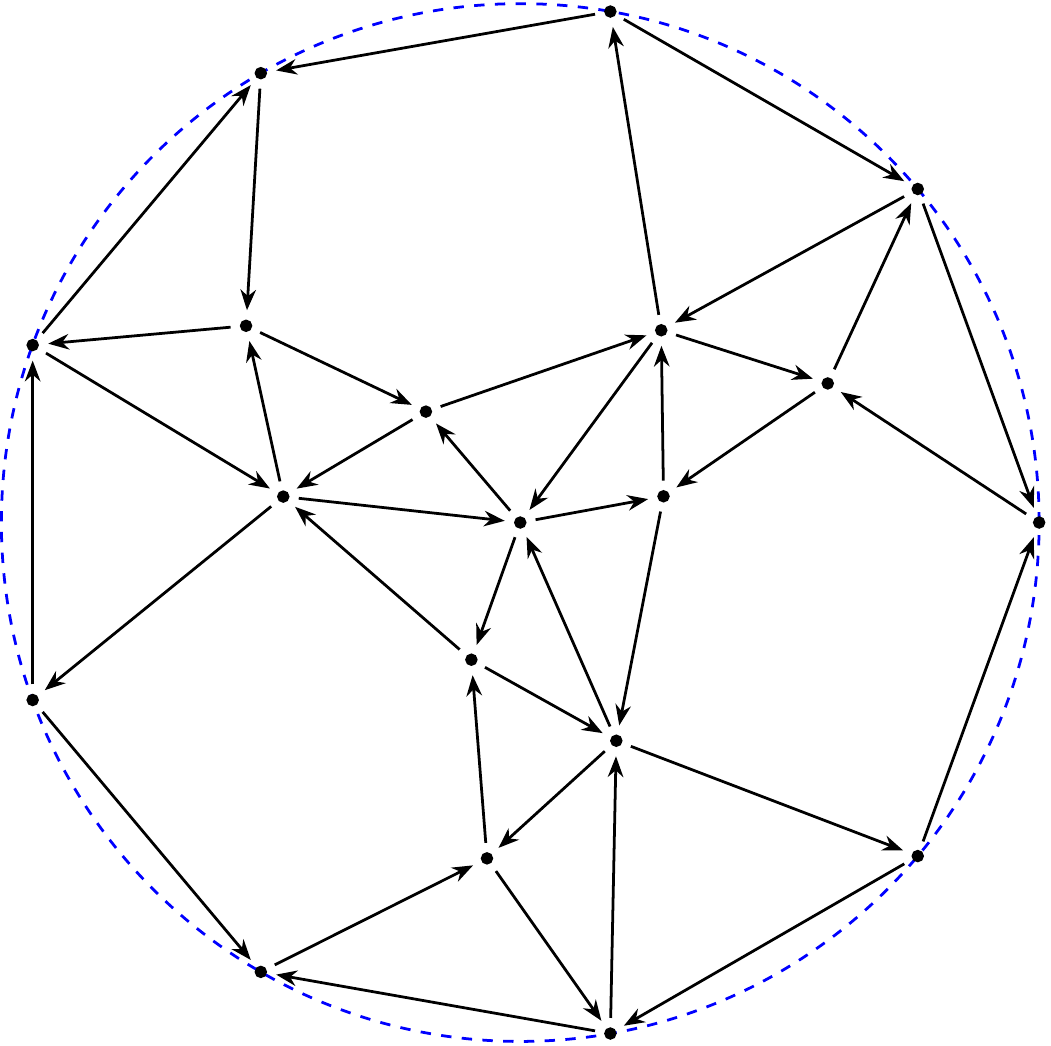}
    \caption{To the left is a (3, 9)-Postnikov diagram, and to the right is its associated quiver, with the vertices on the outer boundary being the frozen vertices.}
    \label{fig:post39_1}
\end{figure}

\begin{figure}[H]
    \includegraphics[scale=0.6, draft=false]{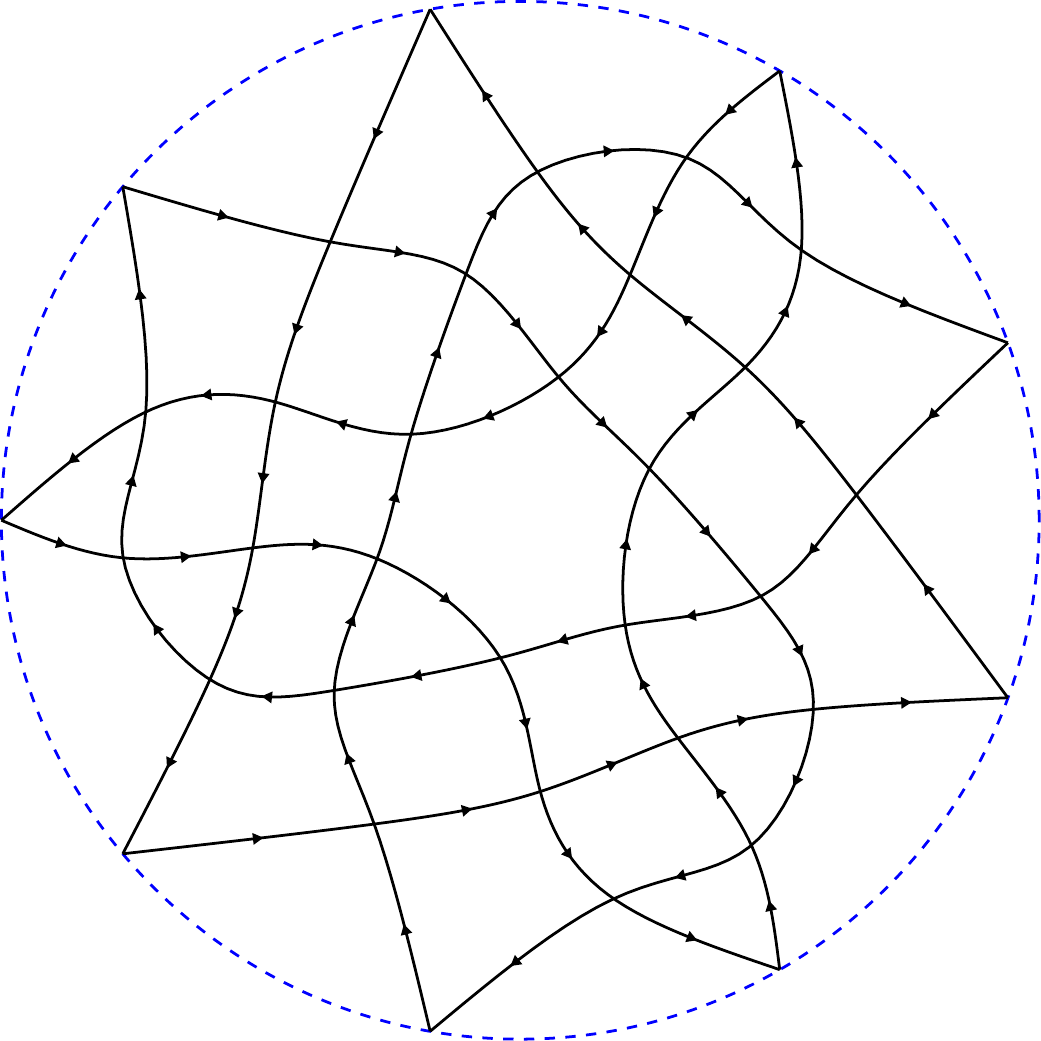}
    \hspace{10pt}
    \includegraphics[scale=0.6, draft=false]{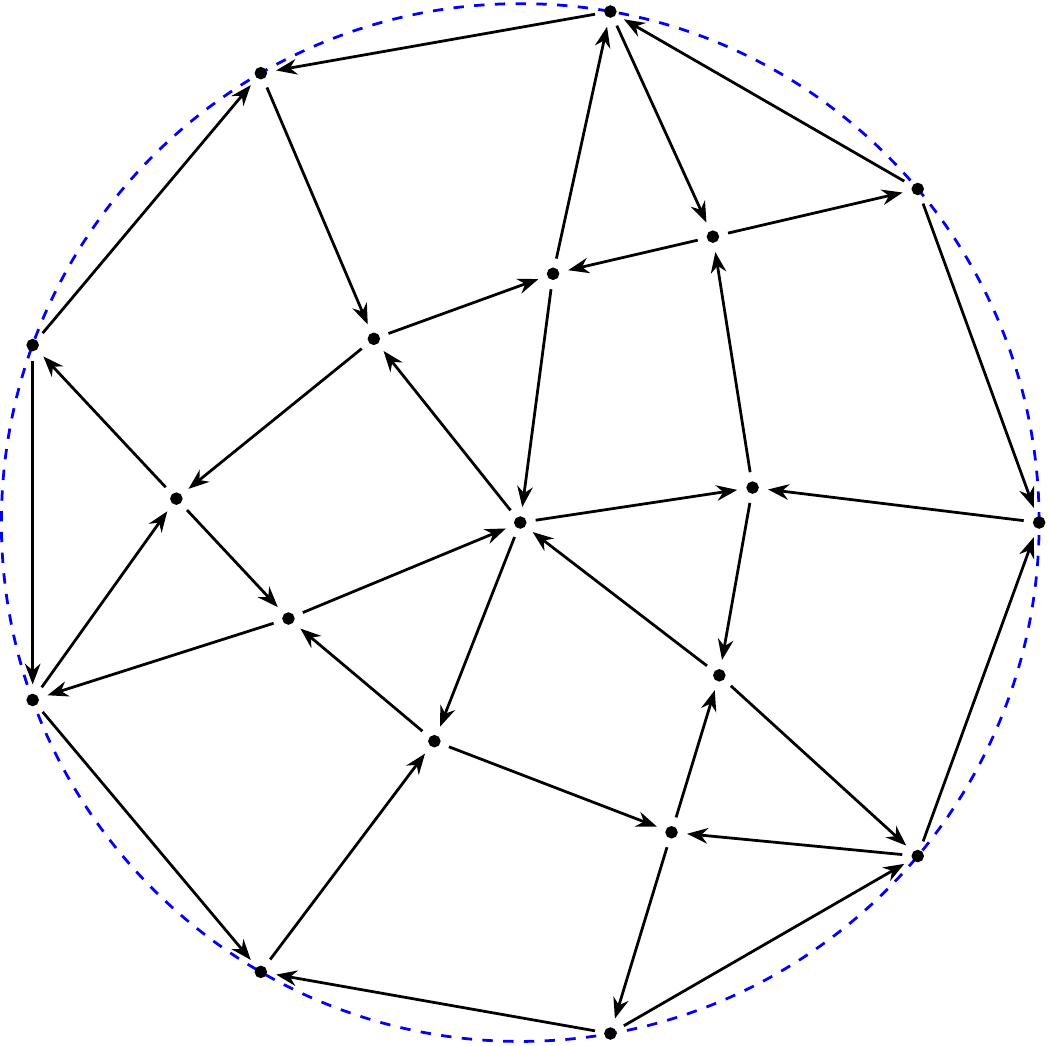}
    \caption{To the left is a (3, 9)-Postnikov diagram, and to the right is its associated quiver, with the vertices on the outer boundary being the frozen vertices.}
    \label{fig:post39_2}
\end{figure}

\textit{Acknowledgement.}
    Thanks to my supervisor Peter Jørgensen for all his insightful advice and corrections.
    Also thanks to Jenny August for her helpful comments on the manuscript.

    \indent This project was supported by grant no. DNRF156 from the Danish National Research Foundation.

\bibliographystyle{amsplain}

\end{document}